\documentclass[a4paper,12pt,reqno]{amsart}
\usepackage{amsmath,amsfonts,amssymb,amsthm,enumerate,multicol}
\usepackage{tikz,graphicx}
\usepackage{subfigure, hyperref}
\usepackage{caption,enumitem,wasysym}
\usepackage{cleveref}
\usepackage{epstopdf}
\usepackage{float,wrapfig}
\usepackage[labelsep=space]{caption}
\usepackage[font=footnotesize]{caption}
\usepackage{xcolor}
\newtheorem{lemma}{Lemma}
\newtheorem{theorem}{Theorem}[section]
\newtheorem{definition}{Definition}[section]
\newtheorem{example}{Example}[section]

\newtheorem{remark}{Remark}[section]

\numberwithin{equation}{section}

\title{Generalized co-polynomials of $R_{II}$ type and associated quadrature rules}
\author{Vinay Shukla$^\dagger$}
\address{$^\dagger$Department of Mathematics (SCSET)\\ Bennett University, Greater Noida-201310, Uttar Pradesh, India}
\email{vinayshukla4321@gmail.com}
\author{A. Swaminathan\, $^{\#\ddagger}$}{\thanks{$^{\#}$Corresponding author}}
\address{$^\ddagger$Department of Mathematics\\ Indian Institute of Technology, Roorkee-247667, Uttarakhand, India}
\email{mathswami@gmail.com, a.swaminathan@ma.iitr.ac.in}
\allowdisplaybreaks
\bigskip

\begin{document}
	
	
	
\begin{abstract}
When the co-recursion and co-dilation in the recurrence relation of certain sequences of orthogonal polynomials are not at the same level, the behaviour of the modified orthogonal polynomials is expected to have different properties compared to the situation of the same level of perturbation. This manuscript attempts to derive structural relations between the perturbed and original $R_{II}$ type orthogonal polynomials. The classical result is improved using a transfer matrix approach. It turns out that the $R_{II}$ fraction with perturbation is the rational spectral transformation of the unperturbed one. The derived notions are used to deduce some consequences for the polynomials orthogonal on the real line. A natural question that arises while dealing with perturbations at different levels, i.e., which perturbation, co-recursion or co-dilation, needs to be performed first, is answered.
\end{abstract}
\maketitle
\markboth{Vinay Shukla and A. Swaminathan}{perturbation at different levels}	

\section{Introduction}\label{Introduction DL}
In the $R_{II}$ type three term recurrence relation
\begin{align}\label{general R2 recurrence DL PP}
&  \mathcal{P}_{n+1}(z) = \rho_n(z-c_n)\mathcal{P}_n(z)-\lambda_n (z-a_n)(z-b_n)\mathcal{P}_{n-1}(z), \quad n\geq 0, \\ 
& \nonumber	\mathcal{P}_{-1}(z) = 0, \qquad \mathcal{P}_0(z) = 1, 
\end{align}
studied in \cite{Esmail masson JAT 1995}, it was shown that if $\lambda_n \neq 0$ and $\mathcal{P}_n(a_n)\mathcal{P}_n(b_n)\neq 0$, for $n\geq 0$, then there exists a rational function $\psi_n(z)=\dfrac{\mathcal{P}_n(z)}{\prod_{j=1}^{n}(z-a_j)(z-b_j)}$ and a linear functional $\mathfrak{N}$ such that the orthogonality relations
\begin{align*}
\mathfrak{N}\left[z^k \psi_n(z) \right] =  0, \quad 0\leq k < n,
\end{align*}
hold \cite[Theorem 3.5]{Esmail masson JAT 1995}. Following \cite{Esmail masson JAT 1995}, the $\mathcal{P}_n(z)$, $n\geq 1$, generated by \eqref{general R2 recurrence DL PP}  will be referred to as $R_{II}$ polynomials (or $R_{II}$ polynomials of first kind). Let $\{\mathcal{Q}_n(z)\}_{n \geq 0}$ be the $R_{II}$ polynomials of second kind satisfying \eqref{general R2 recurrence DL PP} with initial conditions $\mathcal{Q}_0(z) = 0$ and $\mathcal{Q}_1(z) = 1$. They are monic polynomials of degree $n-1$ \cite{swami vinay R2 2022}. 
 
A specific type of $R_{II}$ type recurrence relation is studied in \cite{swami vinay GCRR 2022}
\begin{align}\label{special R2 DL pp}
&\mathcal{P}_{n+1}(x) = \rho_n(x-c_{n})\mathcal{P}_n(x)-\lambda_{n} (x^2+\omega^2)\mathcal{P}_{n-1}(x), \quad n\geq 0, \quad \omega \in \mathbb{R}\backslash\{ 0\}, \\
& \nonumber	\mathcal{P}_{-1}(x) = 0, \qquad \mathcal{P}_0(x) = 1,
\end{align}
where $\{\rho_n \geq 1\}_{n \geq 0}$ and $\{c_n\}_{n \geq 0}$ are sequence of real numbers and $\{\lambda_{n}\}_{n \geq 1}$ is a positive chain sequence. 
Several properties of such $R_{II}$ polynomials (for the case $\rho_n=1$) are obtained in \cite{swami vinay R2 2022} when the recurrence coefficients in \eqref{special R2 DL pp} are subject to perturbation $c_k \rightarrow c_k+\mu_k$ and $\lambda_{k} \rightarrow \nu_{k}\lambda_{k}$, i.e., when both $\{c_n\}_{n \geq 0}$ and $\{\lambda_n\}_{n \geq 1}$ are perturbed at $n=k$. Such polynomials are called co-polynomials of $R_{II}$ type. Structural relation between the perturbed and the original polynomial, a connection with the unit circle and interlacing, inclusion and monotonicity properties of zeros of these polynomials are also investigated in \cite{swami vinay R2 2022}. The computational efficiency of the transfer matrix method over the classical method is also compared \cite{swami vinay R2 2022}. A perturbation in $\{c_n\}_{n \geq 0}$ for $n=k$ and in $\{\lambda_n\}_{n \geq 0}$ for $n=k'$ is a more general case. \par  
In this manuscript, the $R_{II}$ type recurrence relations \eqref{general R2 recurrence DL PP} and \eqref{special R2 DL pp} are analyzed by first modifying the recurrence coefficient $c_n$ for $n=k$ in a co-recursive way and then co-dilating $\lambda_{n}$ for $n=k'$ ($k < k'$), i.e.,
\begin{align}
	c_k &\rightarrow c_k+\mu_k , \qquad \rm(generalized~ co-recursive) \label{co-recursive condition DL}\\
	\lambda_{k'} &\rightarrow \nu_{k'}\lambda_{k'}. \qquad \rm(generalized~ co-dilated) \label{co-dilated condition DL}
\end{align}
The polynomials $\mathcal{P}_{n}(z;\mu_k,\nu_{k'})$ so obtained will be referred to as the first kind generalized co-polynomials of $R_{II}$ type and analogously, the polynomials $\mathcal{Q}_{n}(z;\mu_k,\nu_{k'})$ are called the second kind generalized co-polynomials of $R_{II}$ type. 

In \cite{Zhedanov Biorthogonal JAT 1999}, it is demonstrated that whenever $a_n$ or $b_n$ does not depend on $n$, the $R_{II}$ type recurrence relation \eqref{general R2 recurrence DL PP} can be reduced to a recurrence relation satisfied by orthogonal polynomials on the real line (OPRL), say $\hat{\mathcal{P}}_{n}(x)$, by the transformation 
\begin{align*}
\hat{\mathcal{P}}_{n}(x)=(\gamma x+\delta)^n \mathcal{P}_n \left(\frac{\alpha x+\beta}{\gamma x+\delta}\right)
\end{align*}
under certain restrictions on $\alpha$, $\beta$, $\gamma$ and $\delta$ (see \cite[Section 8, Proposition 3]{Zhedanov Biorthogonal JAT 1999}). More precisely, if $a_n=b_n=a$ in \eqref{general R2 recurrence DL PP}, where $a$ is a constant, and further, if we choose $\alpha = \gamma a$, the recurrence relation \eqref{general R2 recurrence DL PP} reduces to 
\begin{align}\label{OPRL TTRR DL paper}
	&\hat{\mathcal{P}}_{n+1}(x) = \hat{\rho}_n(x-\hat{c}_{n})\hat{\mathcal{P}}_n(x)-\hat{\lambda}_n \hat{\mathcal{P}}_{n-1}(x), \quad \hat{\mathcal{P}}_{-1}(x) = 0, \quad \hat{\mathcal{P}}_0(x) = 1, \quad n \geq 0,
\end{align}
where $\hat{\lambda}_n=\lambda_n (\beta- a \delta)^2$, $\hat{\rho}_n=\rho_n(\alpha-\gamma c_n)$ and $\hat{c}_{n}=\dfrac{\delta c_n-\beta}{\alpha-\gamma c_n}$. Let $\{\hat{\mathcal{Q}}_n(x)\}_{n \geq 0}$ be the second kind OPRL satisfying \eqref{OPRL TTRR DL paper} with initial conditions $\hat{\mathcal{Q}}_0(x) = 0$ and $\hat{\mathcal{Q}}_1(x) = 1$.
\par 

The case is called co-recursive when the first term of the sequence $\{\hat{c}_{n}\}_{n \geq 0}$ is perturbed by adding $\mu_0$, i.e., $\hat{c}_{0} \rightarrow \hat{c}_{0}+\mu_0$ ($k=0$ in \eqref{co-recursive condition DL}), and the polynomials obtained are called co-recursive polynomials \cite{chihara PAMS 1957}. Co-dilation refers to the modification of $\hat{\lambda}_1$ in the sequence $\{\hat{\lambda}_n\}_{n \geq 1}$ by multiplying it with $\nu_1$, i.e., $\hat{\lambda}_1 \rightarrow \nu_1\hat{\lambda}_1$ \cite{Dini thesis 1988}. Note that this corresponds to the case $k'=1$ in \eqref{co-dilated condition DL}. A generalization to these cases is defined by a single modification at the same level in $\{\hat{c}_{n}\}_{n \geq 0}$ and/or $\{\hat{\lambda}_n\}_{n \geq 1}$, say, for $n=k$. The perturbation in $\{\hat{c}_{n}\}_{n \geq 0}$ for $n=k$ is called generalized co-recursive, the perturbation in $\{\hat{\lambda}_n\}_{n \geq 0}$ for $n=k$ is called generalized co-dilated and the condition is called generalized co-modified when both $\{\hat{c}_{n}\}_{n \geq 0}$ and $\{\hat{\lambda}_n\}_{n \geq 1}$ are perturbed for $n=k$ \cite{Paco perturbed recurrence 1990}. The study of the distribution of zeros, the connection between the original and perturbed Stieltjes function, and the derivation of a fourth-order differential equation for co-modified polynomials are carried out in \cite{Paco perturbed recurrence 1990}. Recently, co-polynomials on the real line (COPRL) have been introduced and a transfer matrix approach is used to study the structural relations between the original and perturbed polynomials in \cite{Castillo co-polynomials on real line 2015}. Several new interlacing properties and inequalities involving the zeros of COPRL and original polynomials are presented in \cite{Castillo co-polynomials on real line 2015}.

In this work, the polynomials $\hat{\mathcal{P}}_{n+1}(x;\mu_k,\nu_{k'})$ obtained by introducing \eqref{co-recursive condition DL} and \eqref{co-dilated condition DL} in \eqref{OPRL TTRR DL paper} will be called the first kind generalized COPRL and subsequently, the polynomials $\hat{\mathcal{Q}}_{n+1}(x;\mu_k,\nu_{k'})$ are called the second kind generalized COPRL. The $R_{II}$ type recurrence \eqref{general R2 recurrence DL PP} reduces to a usual three-term recurrence relation under specific assumptions. Hence, the properties of the first and second kind generalized COPRL are derived as a byproduct of results obtained for the $R_{II}$ polynomials. An improved version of Theorem $2.1$ of \cite{Castillo co-polynomials on real line 2015} is provided in this manuscript.

The quadrature rule for $R_{II}$ type recurrence is established in \cite{Bracciali Pereira ranga 2020} for $\rho_n=1$, $n \geq 0$, and $\omega=1$. Its supremacy over the Gauss-Hermite quadrature in evaluating certain integrals is also demonstrated. In this work, we will see the effect of the perturbations considered above on the estimates provided by the quadrature rule. The measure of orthogonality changes when the recurrence coefficients are changed, i.e., the orthogonality measure for the perturbed $R_{II}$ polynomials is not the same as that of the original ones. This problem is addressed in this manuscript, alongwith a method for approximating this new measure. Graphical illustrations are provided to reveal the changes in the orthogonality measure resulting from co-recursion, co-dilation, or co-modification operations. \par 
Further development of this manuscript is outlined as follows: The relations among the perturbed polynomials, the original polynomials and the associated polynomials are obtained using a transfer matrix approach in \Cref{Structural relation}. The perturbed continued fraction, its $(k'+1)$-th tail, and the original continued fraction are related via rational spectral transformation in \Cref{Spectral Transformatiaon of The continued fraction}. \Cref{Connection with the real line} describes the relationship of the results developed in the previous sections for $R_{II}$ polynomials with OPRL. This helped in generalizing several existing results in the literature for the perturbation theory of OPRL. Finally, in \Cref{Numerical quadrature}, it is explained how the estimates provided by an $n$-point rule formulated using the zeros of perturbed $R_{II}$ polynomials can be used to approximate the new measure. Theoretically, it may seem that either doing co-recursion first and then co-dilation, or doing co-dilation first and then co-recursion, does not make any difference. With the help of an illustration, these situations are analyzed from a practical application point of view, and a prescription for which perturbation should be done first is proposed. 
The proofs of the main results are organised in \Cref{Proof of results}.

\section{Structural relation and Spectral transformation}\label{Structural relation and Transfer matrix}
\subsection{Structural relation}\label{Structural relation}
The eventuality of performing the co-recursion first and then the co-dilation is our main focus. Another possibility is to first co-dilate $\lambda_{n}$ for $n=k'$ and then perturb $c_n$ for $n=k$ ($k' < k$) in a co-recursive way. This second aspect can be studied in a similar manner and is hence not discussed in detail in this manuscript. Note that the case $k=k'$ gives the perturbation at the same level. Although the results developed in this manuscript are valid for $k \leq k'$ (or $k' \leq k$), for the sake of generality, we will be dealing with $k < k'$ (or $k' < k$) in whatever follows unless mentioned. The recurrence relation 
\begin{align}\label{Modified RR n < k}
\mathcal{P}_{n+1}(z;\mu_k,\nu_{k'}) &= \rho_n(z-c_n)\mathcal{P}_{n}(z;\mu_k,\nu_{k'})-\lambda_n (z-a_n)(z-b_n)\mathcal{P}_{n-1}(z;\mu_k,\nu_{k'}), 
\end{align}
holds for all $n$, except for $n \in \{k, k' \}$ for which
\begin{equation}
{\footnotesize \begin{aligned}\label{4 recurrence section 2}
\mathcal{P}_{k+1}(z;\mu_k,\nu_{k'}) &= \rho_k(z-c_k-\mu_k)\mathcal{P}_{k}(z;\mu_k,\nu_{k'})- \lambda_k (z-a_{k})(z-b_{k})\mathcal{P}_{k-1}(z;\mu_k,\nu_{k'}),~ \mbox{$n = k$}, \\
\mathcal{P}_{k'+1}(z;\mu_k,\nu_{k'}) &= \rho_{k'}(z-c_{k'})\mathcal{P}_{k'}(z;\mu_k,\nu_{k'})-\nu_{k'} \lambda_{k'} (z-a_{k'})(z-b_{k'})\mathcal{P}_{k'-1}(z;\mu_k,\nu_{k'}), ~ \mbox{$n = k'$},
\end{aligned} }
\end{equation}
holds.

Now, the expressions \eqref{Modified RR n < k} and \eqref{4 recurrence section 2} together with the respective expressions for $\mathcal{Q}_n(z)$ lead to the following result, whose proof is given in \Cref{Proof of results}.

\begin{theorem}\label{Theorem s_k_x DL pp}
The first and second kind generalized co-polynomials of $R_{II}$ type, the original first and second kind $R_{II}$ polynomials, and the associated $R_{II}$ polynomials of order $k+1$ and $k'+1$ satisfy the following structural relation for $n \geq k'$:
{\footnotesize
\begin{align*}
&\mathcal{P}_{n+1}(z;\mu_k,\nu_{k'})	= \mathcal{P}_{n+1}(z) - \mu_k\rho_k \mathcal{P}_k(z)\mathcal{P}^{(k+1)}_{n-k}(z)-(\nu_{k'}-1)\lambda_{k'} (z-a_{k'})(z-b_{k'}) \mathcal{P}_{k'-1}(z)\mathcal{P}^{({k'+1})}_{n-{k'}}(z), \\
&\mathcal{Q}_{n+1}(z;\mu_k,\nu_k)	= \mathcal{Q}_{n+1}(z) - \mu_k\rho_k \mathcal{Q}_k(z)\mathcal{Q}^{(k+1)}_{n-k}(z)-(\nu_{k'}-1)\lambda_{k'} (z-a_{k'})(z-b_{k'}) \mathcal{Q}_{k'-1}(z)\mathcal{Q}^{({k'+1})}_{n-{k'}}(z).
\end{align*}  }
%
\end{theorem}

\begin{remark}
For $j \in \{k, k' \}$, the recurrence relation
	\begin{align*}
		\mathcal{G}^{(j+1)}_{n+1}(z) = \rho_{n+j+1}(z-c_{n+j+1})\mathcal{G}^{(j+1)}_n(z)-\lambda_{n+j+1}(z-a_{n+j+1})(z-b_{n+j+1}) \mathcal{G}^{(j+1)}_{n-1}(z),
	\end{align*}
with initial conditions $\mathcal{G}^{(j+1)}_{-1}(z)=0$ and $\mathcal{G}^{(j+1)}_{0}(z)=1$, defines the first and second kind associated $R_{II}$ polynomials of order $j+1$, appearing in \Cref{Theorem s_k_x DL pp}, whenever $\mathcal{G}^{(j+1)}_{n}(z) = \mathcal{P}^{(j+1)}_{n}(z)$ and $\mathcal{G}^{(j+1)}_{n}(z) = \mathcal{Q}^{(j+1)}_{n}(z)$, respectively. By the Favard theorem \cite{Esmail masson JAT 1995}, there exists a moment functional with respect to which $\{\mathcal{G}^{(j+1)}_{n}(z) \}_{n \geq 0}$ is also a sequence of $R_{II}$ polynomials.
\end{remark}


The next result is an improvement over \Cref{Theorem s_k_x DL pp} in the sense that if we need to compute $\mathcal{P}_{n+1}(z;\mu_k,\nu_{k'})$, \Cref{Theorem s_k_x DL pp} requires the knowledge of the first kind associated $R_{II}$ polynomials of order $k+1$ and $k'+1$, i.e., $\mathcal{P}^{(k+1)}_{n-k}(z)$ and $\mathcal{P}^{(k'+1)}_{n-k'}(z)$, whereas \Cref{transfer matrix theorem R2 pp} requires the information about the polynomials $\mathcal{P}_{n}(z)$ only which is already available.

\begin{theorem}\label{transfer matrix theorem R2 pp}
The relation between $\mathcal{P}_{n}(z;\mu_k,\nu_{k'})$, $\mathcal{P}_{n}(z)$ and their respective second kind polynomial that holds in $\mathbb{C}$ is given as
\begin{align*}
	\prod_{j=1}^{k'}\lambda_j (z-a_j)(z-b_j) \begin{bmatrix}
		\mathcal{P}_{n+1}(z;\mu_k,\nu_{k'}) & \mathcal{P}_{n}(z;\mu_k,\nu_{k'})	\\
		-\mathcal{Q}_{n+1}(z;\mu_k,\nu_{k'}) & -\mathcal{Q}_{n}(z;\mu_k,\nu_{k'})
	\end{bmatrix} & = \mathbf{S}'_k(z) \begin{bmatrix}
		\mathcal{P}_{n+1}(z) &	\mathcal{P}_{n}(z)\\
		-\mathcal{Q}_{n+1}(z) & -\mathcal{Q}_{n}(z)
	\end{bmatrix},
\end{align*}
\begin{align*}
	  {\rm where}, \quad \mathbf{S}'_k(z) = \begin{bmatrix}
		\mathcal{S}'_{11}(z) & \mathcal{S}'_{12}(z)	\\
		\mathcal{S}'_{21}(z) & \mathcal{S}'_{22}(z)
	\end{bmatrix},
\end{align*}
with
\begin{align*}
	&\mathcal{S}'_{11}(z) = \mathfrak{K}'(z)+ \mu_{k}\rho_k\mathcal{P}_{k}\mathcal{Q}_{k}\mathfrak{m}'(z)+(\nu_{k'}-1)\lambda_{k'} (z-a_{k'})(z-b_{k'}) \mathcal{P}_{k'-1}\mathcal{Q}_{k'}, \\
	&\mathcal{S}'_{12}(z) =\mu_{k}\rho_k\mathcal{P}^2_{k}\mathfrak{m}'(z)+(\nu_{k'}-1)\lambda_{k'} (z-a_{k'})(z-b_{k'}) \mathcal{P}_{k'-1}(z)\mathcal{P}_{k'}(z), \\
	&\mathcal{S}'_{21}(z) = -\mu_{k}\rho_k\mathcal{Q}^2_{k}\mathfrak{m}'(z) -(\nu_{k'}-1)\lambda_{k'} (z-a_{k'})(z-b_{k'}) \mathcal{Q}_{k'-1}(z)\mathcal{Q}_{k'}(z), \\
	&\mathcal{S}'_{22}(z) = \mathfrak{K}'(z)- \mu_{k}\rho_k\mathcal{Q}_{k}\mathcal{P}_{k}\mathfrak{m}'(z) -(\nu_{k'}-1)\lambda_{k'} (z-a_{k'})(z-b_{k'}) \mathcal{Q}_{k'-1}\mathcal{P}_{k'}, \\
	&\text{ where} \quad \mathfrak{K}'(z) =\prod_{j=1}^{k'}\lambda_j (z-a_j)(z-b_j), \quad \text{and} \quad \mathfrak{m}'(z)=\prod_{j=k+1}^{k'}\lambda_j(z-a_j)(z-b_j).
\end{align*}
\end{theorem}

The proof of this result uses several notations from the proof of \Cref{Theorem s_k_x DL pp} and is hence given in \Cref{Proof of results}.

\subsection{Spectral transformation of $\mathcal{R}_{II}$-fraction}\label{Spectral Transformatiaon of The continued fraction}
 Associated with \eqref{special R2 DL pp}, the following continued fraction representation can be obtained from \cite[eqn (2.10)]{Esmail masson JAT 1995}
\begin{align}\label{R2 general continued fraction DL pp}
	\mathcal{R}_{II}(z) = \frac{1}{\rho_0(z-c_0)} \mathbin{\genfrac{}{}{0pt}{}{}{-}} \frac{\lambda_1(z-a_1)(z-b_1)}{\rho_1(z-c_1)} \mathbin{\genfrac{}{}{0pt}{}{}{-}} \frac{\lambda_2(z-a_2)(z-b_2)}{\rho_2(z-c_2)} \mathbin{\genfrac{}{}{0pt}{}{}{-}} \mathbin{\genfrac{}{}{0pt}{}{}{\cdots}}.
\end{align}
The above continued fraction terminates for $z=a_k$ or $z=b_k$, $k \geq 1$. Following \cite{Esmail masson JAT 1995}, we call it an $R_{II}$-fraction. The denominator polynomials associated with \eqref{R2 general continued fraction DL pp} are the polynomials $\mathcal{P}_n(z)$, $n\geq 0$, given by \eqref{special R2 DL pp}. They are of degree at most $n$. Furthermore, the polynomials of the second kind, $\mathcal{Q}_n(z)$, $n\geq 1$, are the numerator polynomials associated with \eqref{R2 general continued fraction DL pp}.  
The rational function $\dfrac{\mathcal{Q}_n(z)}{\mathcal{P}_n(z)}$ is the $n$-th convergent of the continued fraction \eqref{R2 general continued fraction DL pp}. In \cite[Theorem 3.7]{Esmail masson JAT 1995}, the existence of a natural Borel measure, say $\beta(z)$, associated with the $R_{II}$-fraction \eqref{R2 general continued fraction DL pp} was also established. \par 
The tail of the continued fraction \eqref{R2 general continued fraction DL pp} obtained after deleting $(k'+1)$ initial terms is given as
\begin{equation}\label{R1_k+1_z continued fraction}
	\mathcal{R}_{II}^{k'+1}(z) = \frac{1}{\rho_{k'+1}(z-c_{k'+1})} \mathbin{\genfrac{}{}{0pt}{}{}{-}} \frac{\lambda_{k'+2}(z-a_{k'+2})(z-b_{k'+2})}{\rho_{k'+2}(z-c_{k'+2})} \mathbin{\genfrac{}{}{0pt}{}{}{-}} \frac{\lambda_{k'+3}(z-a_{k'+3})(z-b_{k'+3})}{\rho_{k'+3}(z-c_{k'+3})} \mathbin{\genfrac{}{}{0pt}{}{}{-}} \mathbin{\genfrac{}{}{0pt}{}{}{\cdots}} .
\end{equation}
We will call such an expression a $(k'+1)$-th tail now onwards.

From \cite[Chapter 4, equation (4.4)]{Chihara book 1978}, we have
\begin{align}\label{chihara identity}
	\frac{A_{n+1}}{B_{n+1}} &= b_0+\frac{a_1}{b_1} \mathbin{\genfrac{}{}{0pt}{}{}{+}} \frac{a_2}{b_2} \mathbin{\genfrac{}{}{0pt}{}{}{+}} \mathbin{\genfrac{}{}{0pt}{}{}{\cdots}} \mathbin{\genfrac{}{}{0pt}{}{}{+}} \frac{a_{n+1}}{b_{n+1}} = \frac{b_{n+1}A_{n}+a_{n+1}A_{n-1}}{b_{n+1}B_{n}+a_{n+1}B_{n-1}}.
\end{align}  
The numerator polynomials of the corresponding continued fraction are $A_n$, and the denominator polynomials are $B_n$. This identity will be used to prove some of the results presented in this section. \par
\begin{definition}\label{def homography mapping}
	A pure rational spectral transformation is referred to as the transformation of a function $u(z)$ \cite{Castillo chapter 2017}, given by
	\begin{align*}
		r(z) \dot{=} \mathbf{A}(z)u(z), \quad \mbox{where} \quad \mathbf{A}(z) = \begin{bmatrix}
			a(z)	&  b(z)\\
			c(z)	&  d(z)
		\end{bmatrix}, \quad a(z)d(z)-b(z)c(z) \neq 0,
	\end{align*} 
	where $a(z), b(z), c(z)$ and $d(z)$ are non-zero polynomials. The $\dot{=}$ notation has been adapted for the homography mapping 
	\begin{align*}
		r(z)=\dfrac{a(z)u(z)+b(z)}{c(z)u(z)+d(z)},
	\end{align*}
	as given in \cite{Castillo chapter 2017}.
\end{definition}
A spectral transformation changes the $R_{II}$-fraction. Precisely, it modifies  $\mathcal{R}_{II}(z)$ given by \eqref{R2 general continued fraction DL pp} associated with the original measure $\beta(z)$ into $\mathcal{R}_{II}(z;\mu_k,\nu_{k'})$ given by \eqref{R1_mu nu} associated with the measure $\alpha(z)$. 
Using \eqref{R1_k+1_z continued fraction}, $\mathcal{R}_{II}(z;\mu_k,\nu_{k'})$ can be written as \\ [2mm]
$\displaystyle \mathcal{R}_{II}(z;\mu_k,\nu_{k'})$
\begin{align}\label{R1_mu nu}
	&= \frac{1}{\rho_0(z-c_0)} \mathbin{\genfrac{}{}{0pt}{}{}{-}} \mathbin{\genfrac{}{}{0pt}{}{}{\cdots}} \mathbin{\genfrac{}{}{0pt}{}{}{-}}  \frac{\lambda_k(z-a_k)(z-b_k)}{\rho_k(z-c_k-\mu_k)} \mathbin{\genfrac{}{}{0pt}{}{}{-}}  \mathbin{\genfrac{}{}{0pt}{}{}{\cdots}} \mathbin{\genfrac{}{}{0pt}{}{}{-}} \frac{\nu_{k'} \lambda_{k'}(z-a_{k'})(z-b_{k'})}{\rho_{k'}( z-c_{k'})} \mathbin{\genfrac{}{}{0pt}{}{}{-}} \mathbin{\genfrac{}{}{0pt}{}{}{\cdots}} \nonumber \\
	&= \frac{1}{\rho_0(z-c_0)} \mathbin{\genfrac{}{}{0pt}{}{}{-}} \mathbin{\genfrac{}{}{0pt}{}{}{\cdots}} \mathbin{\genfrac{}{}{0pt}{}{}{-}}  \frac{\lambda_k(z-a_k)(z-b_k)}{\rho_k(z-c_k-\mu_k)} \mathbin{\genfrac{}{}{0pt}{}{}{-}} \mathbin{\genfrac{}{}{0pt}{}{}{\cdots}} \mathbin{\genfrac{}{}{0pt}{}{}{-}} \nonumber \\
	& \hspace{6cm} \frac{\nu_{k'} \lambda_{k'}(z-a_{k'})(z-b_{k'})}{\rho_{k'}( z-c_{k'})-\lambda_{k'+1}(z-a_{k'+1})(z-b_{k'+1})\mathcal{R}_{II}^{k'+1}(z)}.
\end{align}
Another option, which may also be examined in a similar fashion, is to first co-dilate $\lambda_n$ for $n=k'$ and then perturb $c_n$ for $n=k$ ($k' < k$).

\begin{lemma}\label{Lemma 1}
	The continued fraction $\mathcal{R}_{II}(z;\mu_k,\nu_{k'})$ associated with the generalized co-polynomials of $R_{II}$ type is the rational spectral transformation of its $(k'+1)$-th tail $\mathcal{R}_{II}^{k'+1}(z)$ which can be represented as: 
	\begin{align}\label{R1_mu_nu to R1_k+1}
		\mathcal{R}_{II}(z;\mu_k,\nu_{k'}) \dot{=} \begin{bmatrix}
			\mathcal{A}(z)	&  \mathcal{B}(z)\\
			\mathcal{C}(z) & \mathcal{D}(z)
		\end{bmatrix} \mathcal{R}_{II}^{k'+1}(z),
	\end{align}
	where
	\begin{align*}
		& \mathcal{A}(z) = \lambda_{k'+1}(z-a_{k'+1})(z-b_{k'+1}) [\mathcal{Q}_{k'}(z) - \mu_k\rho_k \mathcal{Q}_k(z)\mathcal{Q}^{(k+1)}_{k'-k-1}(z)], \\
		&\mathcal{B}(z) = -\mathcal{Q}_{k'+1}(z) + \mu_k\rho_k \mathcal{Q}_k(z)\mathcal{Q}^{(k+1)}_{k'-k}(z)+(\nu_{k'}-1)\lambda_{k'} (z-a_{k'})(z-b_{k'}) \mathcal{Q}_{k'-1}(z), \nonumber \\
		& \mathcal{C}(z) = \lambda_{k'+1}(z-a_{k'+1})(z-b_{k'+1}) [\mathcal{P}_{k'}(z) - \mu_k\rho_k \mathcal{P}_k(z)\mathcal{P}^{(k+1)}_{k'-k-1}(z)], \\
		& \mathcal{D}(z) = -\mathcal{P}_{k'+1}(z) + \mu_k\rho_k \mathcal{P}_k(z)\mathcal{P}^{(k+1)}_{k'-k}(z)+(\nu_{k'}-1)\lambda_{k'} (z-a_{k'})(z-b_{k'}) \mathcal{P}_{k'-1}(z). \nonumber 
	\end{align*}
\end{lemma}
\begin{proof}
The continued fraction expansion \eqref{R1_mu nu}, in comparison with \eqref{chihara identity}, gives  
\begin{align*}
\mathcal{R}_{II}(z;\mu_k,\nu_{k'}) 	&=\frac{[\rho_{k'}(z-c_{k'})-\lambda_{k'+1}(z-a_{k'+1})(z-b_{k'+1})\mathcal{R}_{II}^{k'+1}(z)]\mathcal{Q}_{k'}(z;\mu_k)}{[\rho_{k'}(z-c_{k'})-\lambda_{k'+1}(z-a_{k'+1})(z-b_{k'+1})\mathcal{R}_{II}^{k'+1}(z)]\mathcal{P}_{k'}(z;\mu_k)} \\
	&\hspace{6cm} \frac{-\nu_{k'}\lambda_{k'}(z-a_{k'})(z-b_{k'})\mathcal{Q}_{k'-1}(z;\mu_k)}{-\nu_{k'} \lambda_{k'}(z-a_{k'})(z-b_{k'})\mathcal{P}_{k'-1}(z;\mu_k)}\\
	&= \frac{\mathcal{Q}_{k'+1}(z;\mu_k,\nu_{k'}) -\lambda_{k'+1}(z-a_{k'+1})(z-b_{k'+1})\mathcal{R}_{II}^{k'+1}(z) \mathcal{Q}_{k'}(z;\mu_k)}{\mathcal{P}_{k'+1}(z;\mu_k,\nu_{k'}) -\lambda_{k'+1}(z-a_{k'+1})(z-b_{k'+1})\mathcal{R}_{II}^{k'+1}(z) \mathcal{P}_{k'}(z;\mu_k) } \\
	&= 	\frac{\mathcal{A}(z)\mathcal{R}_{II}^{k'+1}(z)+\mathcal{B}(z)}{\mathcal{C}(z)\mathcal{R}_{II}^{k'+1}(z)+\mathcal{D}(z)}.
\end{align*}
Now, \Cref{def homography mapping} and expressions for $\mathcal{P}_{n+1}(z;\mu_k,\nu_{k'})$ and $\mathcal{Q}_{n+1}(z;\mu_k,\nu_{k'})$ given in \Cref{Theorem s_k_x DL pp} proves the result.
\end{proof}

After establishing the preceding result, it appears intruding to investigate the relationship between the $(k'+1)$-th tail and the continued fraction corresponding to the unperturbed polynomial sequence. Consider $\mu_k = 0$ and $\nu_{k'} = 1$. Then, $\mathcal{R}_{II}(z;\mu_k,\nu_{k'})=\mathcal{R}_{II}(z)$. Thus, the relation \eqref{R1_mu_nu to R1_k+1} gives the following result. 
\begin{lemma}\label{Lemma 2}
	$\mathcal{R}_{II}(z)$ and $\mathcal{R}_{II}^{k'+1}(z)$ satisfy the relation
	\begin{align}\label{R1_k+1 to R1_z}
		\lambda_{k'+1}(z-a_{k'+1})(z-b_{k'+1})	\mathcal{R}_{II}^{k'+1}(z) \dot{=} \begin{bmatrix}
			\mathcal{P}_{k'+1}	&  -\mathcal{Q}_{k'+1}\\
			\mathcal{P}_{k'} & -\mathcal{Q}_{k'}
		\end{bmatrix} \mathcal{R}_{II}(z) = \mathbb{F}_{k'+1}\mathcal{R}_{II}(z).
	\end{align}
\end{lemma}	
\begin{proof}
	Putting $\mu_k = 0$ and $\nu_{k'} = 1$, formula \eqref{R1_mu_nu to R1_k+1} takes the form,
{\small	\begin{align*}
		&\mathcal{R}_{II}(z) = \frac{\mathcal{Q}_{k'+1}(z) -\lambda_{k'+1}(z-a_{k'+1})(z-b_{k'+1})\mathcal{R}_{II}^{k'+1}(z) \mathcal{Q}_{k'}(x)}{ \mathcal{P}_{k'+1}(z) -\lambda_{k'+1}(z-a_{k'+1})(z-b_{k'+1})\mathcal{R}_{II}^{k'+1}(z) \mathcal{P}_{k'}(z)}\\
		\Longrightarrow &\lambda_{k'+1}(z-a_{k'+1})(z-b_{k'+1})\mathcal{R}_{II}^{k'+1}(z) \mathcal{Q}_{k'}(z)-\lambda_{k'+1}(z-a_{k'+1})(z-b_{k'+1})\mathcal{R}_{II}^{k'+1}(z) \mathcal{P}_{k'}(z)\mathcal{R}_{II}(z) \\
		&\hspace*{10cm} = \mathcal{Q}_{k'+1}(z)-\mathcal{P}_{k'+1}(z)\mathcal{R}_{II}(z) \\
		\Longrightarrow &\lambda_{k'+1}(z-a_{k'+1})(z-b_{k'+1})\mathcal{R}_{II}^{k'+1}(z)[\mathcal{Q}_{k'}(z)-\mathcal{P}_{k'}(z)\mathcal{R}_{II}(z)] = \mathcal{Q}_{k'+1}(z)-\mathcal{P}_{k'+1}(z)\mathcal{R}_{II}(z) \\
		\Longrightarrow &\lambda_{k'+1}(z-a_{k'+1})(z-b_{k'+1})\mathcal{R}_{II}^{k'+1}(z)= \frac{\mathcal{P}_{k'+1}(z)\mathcal{R}_{II}(z)-\mathcal{Q}_{k'+1}(z)}{\mathcal{P}_{k'}(z)\mathcal{R}_{II}(z)-\mathcal{Q}_{k'}(z)},
	\end{align*}  }
	and hence, the relation \eqref{R1_k+1 to R1_z} is obtained using \Cref{def homography mapping}.
\end{proof}

The above two lemmas are useful in formulating the next result.
\begin{theorem}\label{Spectral tranformation cofactor theorem}
	Let $\mathcal{R}_{II}(z;\mu_k,\nu_{k'})$  be the continued fraction associated with the perturbations \eqref{co-recursive condition DL} and \eqref{co-dilated condition DL}. Then $\mathcal{R}_{II}(z;\mu_k,\nu_{k'})$ is a pure rational spectral transformation of $\mathcal{R}_{II}(z)$ given by
	\begin{align*}
		\mathcal{R}_{II}(z;\mu_k,\nu_{k'}) \dot{=} cof(\mathbf{S}'_k(z))\mathcal{R}_{II}(z),
	\end{align*}
	where cof(.) is the cofactor matrix operator.
\end{theorem}
For additional information on rational spectral transformations, we refer to \cite{Castillo chapter 2017} and references therein.

\subsection{A step further}
The concepts developed in earlier part of this section can be reduced to give results related to special form of $R_{II}$ type recurrence \eqref{special R2 DL pp}. The following results are developed for perturbations \eqref{co-recursive condition DL} and \eqref{co-dilated condition DL} in \eqref{special R2 DL pp} and its second kind couterpart. Another situation that may be similarly investigated is to first perturb $c_n$ for $n=k$ and then co-dilate $\lambda_n$ for $n=k'$ ($k'< k$). It may be noted that the respective proofs of the results in the sequel are similar to the proofs of the earlier results in this section. Hence, only the results are stated without providing proofs.
\begin{theorem}\label{Theorem s_k_x general DL pp}
The structural relation between the perturbed $R_{II}$ polynomials, the original $R_{II}$ polynomials and the associated $R_{II}$ polynomials of order $k+1$ and $k'+1$ for $n \geq k'$ is given by 
{\small \begin{align*}
	{\mathcal{P}}_{n+1}(x;\mu_k,\nu_{k'}) &= \mathcal{P}_{n+1}(x) - \mu_k\rho_k \mathcal{P}_k(x)\mathcal{P}^{(k+1)}_{n-k}(x)-(\nu_{k'}-1)\lambda_{k'} (x^2+\omega^2) \mathcal{P}_{k'-1}(x)\mathcal{P}^{({k'+1})}_{n-{k'}}(x),  \\
	{\mathcal{Q}}_{n+1}(x;\mu_k,\nu_{k'}) &= \mathcal{Q}_{n+1}(x) - \mu_k\rho_k \mathcal{Q}_k(x)\mathcal{Q}^{(k+1)}_{n-k}(x)-(\nu_{k'}-1)\lambda_{k'} (x^2+\omega^2) \mathcal{Q}_{k'-1}(x)\mathcal{Q}^{({k'+1})}_{n-{k'}}(x). 
\end{align*} }
\end{theorem}

\begin{remark}
	For $j \in \{k, k' \}$, the recurrence relation
	\begin{align*}
		\mathcal{G}^{(j+1)}_{n+1}(x) = \rho_{n+j+1}(x-c_{n+j+1})\mathcal{G}^{(j+1)}_n(x)-\lambda_{n+j+1} (x^2+\omega^2)\mathcal{G}^{(j+1)}_{n-1}(x), \quad n\geq 0,
	\end{align*}
	with initial conditions $\mathcal{G}^{(j+1)}_{-1}(x)=0$ and $\mathcal{G}^{(j+1)}_{0}(x)=1$, defines the first and second kind associated polynomials of order $j+1$ whenever $\mathcal{G}^{(j+1)}_{n}(x) = \mathcal{P}^{(j+1)}_{n}(x)$ and $\mathcal{G}^{(j+1)}_{n}(x) = \mathcal{Q}^{(j+1)}_{n}(x)$, respectively. By the Favard theorem \cite{Esmail masson JAT 1995}, there exists a moment functional with respect to which $\{\mathcal{G}^{(j+1)}_{n}(x) \}_{n \geq 0}$ is also a sequence of $R_{II}$ polynomials.
\end{remark}

\begin{theorem}\label{transfer matrix theorem}
The polynomial matrix $\mathbb{F}_{n+1}(x;\mu_k,\nu_{k'})$ containing generalized co-polynomials of $R_{II}$ type can be obtained from the polynomial matrix $\mathbb{F}_{n+1}(x)$ of original $R_{II}$ polynomials in the following way:
\begin{align*}
	(x^2+\omega^2)^{k'}\prod_{j=1}^{k'}\lambda_j \begin{bmatrix}
		\mathcal{P}_{n+1}(x;\mu_k,\nu_{k'}) & \mathcal{P}_{n}(x;\mu_k,\nu_{k'})	\\
		-\mathcal{Q}_{n+1}(x;\mu_k,\nu_{k'}) & -\mathcal{Q}_{n}(x;\mu_k,\nu_{k'})
	\end{bmatrix} & = \mathbf{S}_k(x) \begin{bmatrix}
		\mathcal{P}_{n+1}(x) &	\mathcal{P}_{n}(x)\\
		-\mathcal{Q}_{n+1}(x) & -\mathcal{Q}_{n}(x)
	\end{bmatrix},
\end{align*}
or equivalently,
\begin{align*}
	\mathfrak{K}(x) \mathbb{F}^T_{n+1}(x;\mu_k,\nu_{k'}) &= \mathbf{S}_k(x) \mathbb{F}_{n+1}(x), \quad
	{\rm where}, \quad \mathbf{S}_k(x) = \begin{bmatrix}
		\mathcal{S}_{11}(x) & \mathcal{S}_{12}(x)	\\
		\mathcal{S}_{21}(x) & \mathcal{S}_{22}(x)
	\end{bmatrix},
\end{align*}
with
\begin{align*}
	&\mathcal{S}_{11}(x) = \mathfrak{K}(x)+ \mu_{k}\rho_k\mathcal{P}_{k}(x)\mathcal{Q}_{k}(x)\mathfrak{m}(x)+(\nu_{k'}-1)\lambda_{k'} (x^2+\omega^2) \mathcal{P}_{k'-1}(x)\mathcal{Q}_{k'}(x), \\
	&\mathcal{S}_{12}(x) =\mu_{k}\rho_k\mathcal{P}^2_{k}(x)\mathfrak{m}(x)+(\nu_{k'}-1)\lambda_{k'} (x^2+\omega^2) \mathcal{P}_{k'-1}(x)\mathcal{P}_{k'}(x), \\
	&\mathcal{S}_{21}(x) = -\mu_{k}\rho_k\mathcal{Q}^2_{k}(x)\mathfrak{m}(x) -(\nu_{k'}-1)\lambda_{k'} (x^2+\omega^2) \mathcal{Q}_{k'-1}(x)\mathcal{Q}_{k'}(x), \\
	&\mathcal{S}_{22}(x) = \mathfrak{K}(x)- \mu_{k}\rho_k\mathcal{Q}_{k}(x)\mathcal{P}_{k}(x)\mathfrak{m}(x) -(\nu_{k'}-1)\lambda_{k'} (x^2+\omega^2) \mathcal{Q}_{k'-1}(x)\mathcal{P}_{k'}(x) \\
	&\text{where} \quad \mathfrak{K}(x) =(x^2+\omega^2)^{k'}\prod_{j=1}^{k'}\lambda_j \quad \text{and} \quad  \mathfrak{m}(x)=(x^2+\omega^2)^{k'-k-1}\prod_{j=k+1}^{k'}\lambda_j.
\end{align*}
\end{theorem}

\begin{remark}
	Assume $(x^2+\omega^2)^{m-n}\prod_{j=n}^{m}\lambda_j = 1$ whenever $m<n$ and observe that for $k=k'$ and $\rho_k=1$, the expression for $\mathcal{P}_{n+1}(x;\mu_k,\nu_{k'})$ becomes
	\begin{align*}
		\mathcal{P}_{n+1}(x;\mu_k,\nu_{k})&= \mathcal{P}_{n+1}(x) - [\mu_k \mathcal{P}_k(x)+(\nu_{k}-1)\lambda_{k} (x^2+\omega^2) \mathcal{P}_{k-1}(x)]\mathcal{P}^{({k+1})}_{n-{k}}(x) \\
		&=\mathcal{P}_{n+1}(x)-\mathcal{S}_k(x)\mathcal{P}^{({k+1})}_{n-{k}}(x),
	\end{align*}
	where $\mathcal{S}_k(x)$ is the one defined in \cite[Theorem 2.1]{swami vinay R2 2022}. Furthermore, the matrix $\mathbf{S}_k(x)$ is transformed into the matrix $\mathbf{N}_k$ defined in \cite[Theorem 3.1]{swami vinay R2 2022}. Thus, \Cref{Theorem s_k_x DL pp} and \Cref{transfer matrix theorem R2 pp} are generalizations of Theorem $2.1$ and Theorem $3.1$, respectively, given in \cite{swami vinay R2 2022}.
\end{remark}

\begin{remark}
It is worth noting that \Cref{transfer matrix theorem R2 pp} holds true for the entire complex plane $\mathbb{C}$, whereas \Cref{transfer matrix theorem} holds true in $\mathbb{R}$. Further, it is easy to verify that for $a_j=i\omega$, $b_j=-i\omega$, $\forall ~ j$, and $z \in \mathbb{R}$, \Cref{transfer matrix theorem R2 pp} implies \Cref{transfer matrix theorem}.
\end{remark}
The following continued fraction expansions for $\mathcal{R}_{II}(x)$ \cite{Esmail masson JAT 1995} and $\mathcal{R}_{II}(x;\mu_k,\nu_{k'})$ are used to establish subsequent results.
\begin{equation}\label{R2_x continued fraction}
	\mathcal{R}_{II}(x) = \frac{1}{\rho_0(x-c_0)} \mathbin{\genfrac{}{}{0pt}{}{}{-}} \frac{\lambda_1(x^2+\omega^2)}{\rho_1(x-c_1)} \mathbin{\genfrac{}{}{0pt}{}{}{-}} \frac{\lambda_2(x^2+\omega^2)}{\rho_2(x-c_2)} \mathbin{\genfrac{}{}{0pt}{}{}{-}} \mathbin{\genfrac{}{}{0pt}{}{}{\cdots}}.
\end{equation}
\\ [2mm]
$ \displaystyle
\mathcal{R}_{II}(x;\mu_k,\nu_{k'})
$
\begin{align}
	&= \frac{1}{\rho_0(x-c_0)} \mathbin{\genfrac{}{}{0pt}{}{}{-}} \mathbin{\genfrac{}{}{0pt}{}{}{\cdots}} \mathbin{\genfrac{}{}{0pt}{}{}{-}}  \frac{\lambda_k(x^2+\omega^2)}{\rho_k(x-c_k-\mu_k)} \mathbin{\genfrac{}{}{0pt}{}{}{-}}  \mathbin{\genfrac{}{}{0pt}{}{}{\cdots}} \mathbin{\genfrac{}{}{0pt}{}{}{-}} \frac{\nu_{k'} \lambda_{k'}(x^2+\omega^2)}{\rho_{k'}(x-c_{k'})} \mathbin{\genfrac{}{}{0pt}{}{}{-}} \frac{\lambda_{k'+1}(x^2+\omega^2)}{\rho_{k'+1}(x-c_{k'+1})} \mathbin{\genfrac{}{}{0pt}{}{}{-}} \mathbin{\genfrac{}{}{0pt}{}{}{\cdots}} \nonumber \\
	&= \frac{1}{\rho_0(x-c_0)} \mathbin{\genfrac{}{}{0pt}{}{}{-}} \mathbin{\genfrac{}{}{0pt}{}{}{\cdots}} \mathbin{\genfrac{}{}{0pt}{}{}{-}}  \frac{\lambda_k(x^2+\omega^2)}{\rho_k(x-c_k-\mu_k)} \mathbin{\genfrac{}{}{0pt}{}{}{-}} \mathbin{\genfrac{}{}{0pt}{}{}{\cdots}} \mathbin{\genfrac{}{}{0pt}{}{}{-}} \frac{\nu_{k'} \lambda_{k'}(x^2+\omega^2)}{\rho_{k'}(x-c_{k'})-\lambda_{k'+1}(x^2+\omega^2)\mathcal{R}_{II}^{k'+1}(x)}.
\end{align}

Note that the infinite continued fraction \eqref{R2_x continued fraction} terminates when $x=\pm i\omega$.

\begin{theorem}\label{R2_mu_nu to DL}
$\mathcal{R}_{II}(x;\mu_k,\nu_{k'})$ defines a rational spectral transformation of $\mathcal{R}_{II}(x)$ as
\begin{align*}
	\mathcal{R}_{II}(x;\mu_k,\nu_{k'}) \dot{=} cof(\mathbf{S}_k(x))\mathcal{R}_{II}(x),
\end{align*}
where $\mathbf{S}_{k}(x)$ is as given in \Cref{transfer matrix theorem}.
\end{theorem}
\subsection{Connection with the OPRL}\label{Connection with the real line}
As mentioned in the \Cref{Introduction DL}, we recover analogous properties for OPRL from those developed for ${R}_{II}$ polynomials. The results so obtained proved to generalize several existing results in the literature for the perturbation theory of OPRL.
\begin{theorem}\label{Theorem s_k_x OPRL DL pp}
The following structural relations between the first and second kind generalized COPRL $\hat{\mathcal{P}}_{n}(x;\mu_k,\nu_{k'})$ and $\hat{\mathcal{Q}}_{n}(x;\mu_k,\nu_{k'})$, the original OPRL $\hat{\mathcal{P}}_n(x)$ and $\hat{\mathcal{Q}}_n(x)$ satisfying \eqref{OPRL TTRR DL paper}, and the first and second kind associated OPRL of order $k+1$ and $k'+1$ hold for $n \geq k'$:
\begin{align*}
\hat{\mathcal{P}}_{n+1}(x;\mu_k,\nu_{k'}) &= \hat{\mathcal{P}}_{n+1}(x) - \mu_k\hat{\rho}_k \hat{\mathcal{P}}_k(x)\hat{\mathcal{P}}^{(k+1)}_{n-k}(x)-(\nu_{k'}-1)\hat{\lambda}_{k'}  \hat{\mathcal{P}}_{k'-1}(x)\hat{\mathcal{P}}^{({k'+1})}_{n-{k'}}(x),  \\
\hat{\mathcal{Q}}_{n+1}(x;\mu_k,\nu_{k'}) &= \hat{\mathcal{Q}}_{n+1}(x) - \mu_k\hat{\rho}_k \hat{\mathcal{Q}}_k(x)\hat{\mathcal{Q}}^{(k+1)}_{n-k}(x)-(\nu_{k'}-1)\hat{\lambda}_{k'}  \hat{\mathcal{Q}}_{k'-1}(x)\hat{\mathcal{Q}}^{({k'+1})}_{n-{k'}}(x). 
\end{align*}
\end{theorem}
\Cref{Theorem s_k_x OPRL DL pp} generalizes several results given in \cite{Paco perturbed recurrence 1990}. For example, the case $\mu_{k}=0$ reduces to \cite[Section 2.1]{Paco perturbed recurrence 1990} and $\nu_{k'}=1$ reduces to \cite[Section 2.2]{Paco perturbed recurrence 1990}. Similarly, the particular case $k=k'$ is addressed in \cite[Section 2.3]{Paco perturbed recurrence 1990}.
\begin{remark}
The first and second kind associated OPRL of order $j+1$ for $j \in \{k, k' \}$ can be obtained from the relation
\begin{align*}
\hat{\mathcal{G}}^{(j+1)}_{n+1}(x) =\hat{\rho}_{n+j+1} (x-\hat{c}_{n+j+1})\hat{\mathcal{G}}^{(j+1)}_n(x)-\hat{\lambda}_{n+j+1} \hat{\mathcal{G}}^{(j+1)}_{n-1}(x), \quad n\geq 0,
\end{align*}
with initial conditions $\hat{\mathcal{G}}^{(j+1)}_{-1}(x)=0$ and $\hat{\mathcal{G}}^{(j+1)}_{0}(x)=1$ by substituting $\hat{\mathcal{G}}^{(j+1)}_{n}(x)=\hat{\mathcal{P}}^{(j+1)}_{n}(x)$ and $\hat{\mathcal{G}}^{(j+1)}_{n}(x)=\hat{\mathcal{Q}}^{(j+1)}_{n}(x)$, respectively.
\end{remark}
Now, we consider the following theorem given in \cite{Castillo co-polynomials on real line 2015} which can be improved using the developments given above.
\begin{theorem}\cite[Theorem 2.1]{Castillo co-polynomials on real line 2015}\label{Correction Original}
For $x \in \mathbb{R}\backslash X$, the following relations hold:
\begin{align}
\mathcal{P}_{n}(x;\mu_{k+1},\nu_k)&=\mathcal{P}_{n}(x), \qquad \mbox{$n \leq k$},\nonumber \\
\mathcal{P}_{n}(x;\mu_{k+1},\nu_k) &= \mathcal{P}_{n}(x) - \mathcal{W}_k(x)\mathcal{P}^{(k)}_{n-k}(x), \qquad \mbox{$n > k$}, \label{Ambiguity 2}
\end{align}
where $\mathcal{W}_k(x)=\mu_{k+1} \mathcal{P}_k(x)+(\nu_k-1)\lambda_k \mathcal{P}_{k-1}(x)$ and $X$ is the set of zeros of $\mathcal{P}_{k-1}(x)$.
\end{theorem}
The following shifted forms of recurrence relation for OPRL and associated polynomials are used in \cite{Castillo co-polynomials on real line 2015}:
\begin{align}
&\mathcal{P}_{n+1}(x) = (x-c_{n+1})\mathcal{P}_n(x)-\lambda_n \mathcal{P}_{n-1}(x), \quad \mathcal{P}_{-1}(x) = 0, \quad \mathcal{P}_0(x) = 1, \quad n \geq 0, \label{Correction recurrence}\\
&\mathcal{P}^{(k)}_{n+1}(x) = (x-c_{n+k+1})\mathcal{P}^{(k)}_n(x)-\lambda_{n+k} \mathcal{P}^{(k)}_{n-1}(x), \quad \mathcal{P}^{(k)}_{-1}(x) = 0, \quad \mathcal{P}^{(k)}_0(x) = 1. \label{Correction associated}
\end{align}
Let us calculate $\mathcal{P}_{k+1}(x;\mu_{k+1},\nu_k)$ in two ways: 
\begin{enumerate}
\item Using \Cref{Correction Original} and relations \eqref{Correction recurrence} and \eqref{Correction associated}, we obtain
\begin{align}
\mathcal{P}_{k+1}(x;\mu_{k+1},\nu_k)=\mathcal{P}_{k+1}(x)-\mathcal{W}_k(x)\mathcal{P}^{(k)}_{1}(x)=\mathcal{P}_{k+1}(x)-\mathcal{W}_k(x)(x-c_{k+1}). \label{Ambiguity method 1}
\end{align}
\item A direct computation from the recurrence relation \eqref{Correction recurrence} shows that
\begin{align}
\mathcal{P}_{k+1}(x;\mu_{k+1},\nu_k) &= (x-c_{k+1}-\mu_{k+1})\mathcal{P}_{k}(x)-\nu_k \lambda_k \mathcal{P}_{k-1}(x)\nonumber	\\
&=(x-c_{k+1})\mathcal{P}_{k}(x)-\nu_k\mathcal{P}_{k-1}(x)-\mu_{k+1}\mathcal{P}_{k}(x)-(\nu_k-1)
\lambda_k \mathcal{P}_{k-1}(x) \nonumber\\
&=\mathcal{P}_{k+1}(x)-\mathcal{W}_k(x) \label{Ambiguity method 2}
\end{align}
\end{enumerate}
While both methods should produce the same result, an extra $(x-c_{k+1})$ is involved in \eqref{Ambiguity method 1}. Although the underlying concepts developed in \cite{Paco perturbed recurrence 1990} and \cite{Castillo co-polynomials on real line 2015} are the same, the results generated in Section $(2.3)$ of \cite{Paco perturbed recurrence 1990} and Theorem $2.1$ of \cite{Castillo co-polynomials on real line 2015} are easily seen to have a difference. This can be corrected using \Cref{Theorem s_k_x OPRL DL pp} for $k=k'$, and thus Theorem $2.1$ of \cite{Castillo co-polynomials on real line 2015} takes the following form:
\begin{theorem}\label{Corrected theorem}
The following relations hold in $\mathbb{R}$:
\begin{align*}
\mathcal{P}_{n}(x;\mu_{k+1},\nu_k)&=\mathcal{P}_{n}(x), \qquad \mbox{$n \leq k$}, \\
\mathcal{P}_{n}(x;\mu_{k+1},\nu_k) &= \mathcal{P}_{n}(x) - \mathcal{W}_k(x)\mathcal{P}^{(k+1)}_{n-(k+1)}(x), \qquad \mbox{$n > k$}, 
\end{align*}
where $\mathcal{W}_k(x)=\mu_{k+1} \mathcal{P}_k(x)+(\nu_k-1)\lambda_k \mathcal{P}_{k-1}(x)$.
\end{theorem}
If we compute $\mathcal{P}_{k+1}(x;\mu_{k+1},\nu_k)$ using \Cref{Corrected theorem}, the expression obtained coincides with \eqref{Ambiguity method 2}. Further, \Cref{Corrected theorem} can easily be seen to be consistent with \cite[Section 2.3]{Paco perturbed recurrence 1990}. Furthermore, \Cref{Corrected theorem} is an outcome of \Cref{Theorem s_k_x OPRL DL pp}, which is proven using transfer matrices and thus holds in $\mathbb{R}$. Therefore, \Cref{Corrected theorem} also holds in $\mathbb{R}$, whereas \Cref{Correction Original} holds true for $\mathbb{R}\backslash X$ only. With this point of view also, \Cref{Corrected theorem} can be seen as an improvement over \Cref{Correction Original}.

\begin{theorem}\label{transfer matrix theorem OPRL R2 pp}
The polynomial matrix $\hat{\mathbb{F}}_{n+1}(x;\mu_k,\nu_{k'})$ of generalized COPRL can be obtained by simply multiplying the transfer matrix $\hat{\mathbf{S}}_k(x)$ with the polynomial matrix $\hat{\mathbb{F}}_{n+1}(x)$ of original OPRL, i.e., 
\begin{align*}
\prod_{j=1}^{k'}\hat{\lambda}_j \begin{bmatrix}
	\hat{\mathcal{P}}_{n+1}(x;\mu_k,\nu_{k'}) & \hat{\mathcal{P}}_{n}(x;\mu_k,\nu_{k'})	\\
	-\hat{\mathcal{Q}}_{n+1}(x;\mu_k,\nu_{k'}) & -\hat{\mathcal{Q}}_{n}(x;\mu_k,\nu_{k'})
\end{bmatrix} & = \hat{\mathbf{S}}_k(x) \begin{bmatrix}
	\hat{\mathcal{P}}_{n+1}(x) &	\hat{\mathcal{P}}_{n}(x)\\
	-\hat{\mathcal{Q}}_{n+1}(x) & -\hat{\mathcal{Q}}_{n}(x)
\end{bmatrix},
\end{align*}
or equivalently, 
\begin{align*}
\hat{\mathfrak{K}} \hat{\mathbb{F}}^T_{n+1}(x;\mu_k,\nu_{k'}) &= \hat{\mathbf{S}}_k(x) \hat{\mathbb{F}}_{n+1}(x), \quad {\rm where}, \quad \hat{\mathbf{S}}_k(x)= \begin{bmatrix}
	\hat{\mathcal{S}}_{11}(x) & \hat{\mathcal{S}}_{12}(x)	\\
	\hat{\mathcal{S}}_{21}(x) & \hat{\mathcal{S}}_{22}(x)
\end{bmatrix},
\end{align*}
with
\begin{align*}
&\hat{\mathcal{S}}_{11}(x) = \hat{\mathfrak{K}}+ \mu_{k}\hat{\rho}_k\hat{\mathcal{P}}_{k}(x)\hat{\mathcal{Q}}_{k}(x)\hat{\mathfrak{m}}+(\nu_{k'}-1)\hat{\lambda}_{k'} \hat{\mathcal{P}}_{k'-1}(x)\hat{\mathcal{Q}}_{k'}(x), \\
&\hat{\mathcal{S}}_{12}(x) =\mu_{k}\hat{\rho}_k\hat{\mathcal{P}}^2_{k}(x)\hat{\mathfrak{m}}+(\nu_{k'}-1)\hat{\lambda}_{k'} \hat{\mathcal{P}}_{k'-1}(x)\hat{\mathcal{P}}_{k'}(x), \\
&\hat{\mathcal{S}}_{21}(x) = -\mu_{k}\hat{\rho}_k\hat{\mathcal{Q}}^2_{k}(x)\hat{\mathfrak{m}} -(\nu_{k'}-1)\hat{\lambda}_{k'} \hat{\mathcal{Q}}_{k'-1}(x)\hat{\mathcal{Q}}_{k'}(x), \\
&\hat{\mathcal{S}}_{22}(x) = \hat{\mathfrak{K}}- \mu_{k}\hat{\rho}_k\hat{\mathcal{Q}}_{k}(x)\hat{\mathcal{P}}_{k}(x)\hat{\mathfrak{m}}-(\nu_{k'}-1)\hat{\lambda}_{k'} \hat{\mathcal{Q}}_{k'-1}(x)\hat{\mathcal{P}}_{k'}(x),\\
&\text{where} \quad \hat{\mathfrak{K}} =\prod_{j=1}^{k'}\hat{\lambda}_j , \quad \text{and} \quad  \hat{\mathfrak{m}}=\prod_{j=k+1}^{k'}\hat{\lambda}_j.
\end{align*}
\end{theorem}

\begin{remark}
For $m<n$, we assume $\prod_{j=n}^{m}\hat{\lambda}_j=1$. The matrix $\hat{\mathbf{S}}_k(x)$ becomes the matrix $\mathbf{M}_k$ defined in \cite[Theorem 3.1]{Castillo co-polynomials on real line 2015} for $k=k'$ in \Cref{transfer matrix theorem OPRL R2 pp}. Thus, \Cref{transfer matrix theorem OPRL R2 pp} is a generalisation of \cite[Theorem 3.1]{Castillo co-polynomials on real line 2015}.
\end{remark}


The polynomials $\hat{\mathcal{P}}_{n}(x)$ and $\hat{\mathcal{Q}}_{n}(x)$ are the denominator and numerator polynomials of the continued fraction $\hat{\mathcal{R}}(x)$  \cite{Chihara book 1978} whereas generalized COPRL $\hat{\mathcal{P}}_{n}(x;\mu_k,\nu_{k'})$ and $\hat{\mathcal{Q}}_{n}(x;\mu_k,\nu_{k'})$ are the denominator and numerator polynomials of the continued fraction $\hat{\mathcal{R}}(x;\mu_k,\nu_{k'})$ \cite{chihara PAMS 1957}. These continued fractions $\hat{\mathcal{R}}(x)$ and $\hat{\mathcal{R}}(x;\mu_k,\nu_{k'})$ are given by
\begin{align}\label{OPRL general continued fraction DL pp}
\hat{\mathcal{R}}(x) = \frac{1}{\hat{\rho}_0(x-\hat{c}_0)} \mathbin{\genfrac{}{}{0pt}{}{}{-}} \frac{\hat{\lambda}_1}{\hat{\rho}_1(x-\hat{c}_1)} \mathbin{\genfrac{}{}{0pt}{}{}{-}} \frac{\hat{\lambda}_2}{\hat{\rho}_2(x-\hat{c}_2)} \mathbin{\genfrac{}{}{0pt}{}{}{-}} \mathbin{\genfrac{}{}{0pt}{}{}{\cdots}}.
\end{align}
\begin{align*}
\hat{\mathcal{R}}(x;\mu_k,\nu_{k'})&= \frac{1}{\hat{\rho}_0(x-\hat{c}_0)} \mathbin{\genfrac{}{}{0pt}{}{}{-}} \mathbin{\genfrac{}{}{0pt}{}{}{\cdots}} \mathbin{\genfrac{}{}{0pt}{}{}{-}}  \frac{\hat{\lambda}_k}{\hat{\rho}_k(x-\hat{c}_k-\mu_k)} \mathbin{\genfrac{}{}{0pt}{}{}{-}}  \mathbin{\genfrac{}{}{0pt}{}{}{\cdots}} \mathbin{\genfrac{}{}{0pt}{}{}{-}} \frac{\nu_{k'} \hat{\lambda}_{k'}}{\hat{\rho}_{k'}(x-\hat{c}_{k'})} \mathbin{\genfrac{}{}{0pt}{}{}{-}} \frac{\hat{\lambda}_{k'+1}}{x-\hat{c}_{k'+1}} \mathbin{\genfrac{}{}{0pt}{}{}{-}} \mathbin{\genfrac{}{}{0pt}{}{}{\cdots}} \nonumber \\
&= \frac{1}{\hat{\rho}_0(x-\hat{c}_0)} \mathbin{\genfrac{}{}{0pt}{}{}{-}} \mathbin{\genfrac{}{}{0pt}{}{}{\cdots}} \mathbin{\genfrac{}{}{0pt}{}{}{-}}  \frac{\hat{\lambda}_k}{\hat{\rho}_k(x-\hat{c}_k-\mu_k)} \mathbin{\genfrac{}{}{0pt}{}{}{-}} \mathbin{\genfrac{}{}{0pt}{}{}{\cdots}} \mathbin{\genfrac{}{}{0pt}{}{}{-}} \frac{\nu_{k'} \hat{\lambda}_{k'}}{\hat{\rho}_{k'}(x-\hat{c}_{k'})-\hat{\lambda}_{k'+1}\hat{\mathcal{R}}^{k'+1}(x)}, \\
{\rm where}\quad \hat{\mathcal{R}}^{k'+1}(x) &= \frac{1}{\hat{\rho}_{k'+1}(x-\hat{c}_{k'+1})} \mathbin{\genfrac{}{}{0pt}{}{}{-}} \frac{\hat{\lambda}_{k'+2}}{\hat{\rho}_{k'+2}(x-\hat{c}_{k'+2})} \mathbin{\genfrac{}{}{0pt}{}{}{-}} \frac{\hat{\lambda}_{k'+3}}{\hat{\rho}_{k'+3}(x-\hat{c}_{k'+3})} \mathbin{\genfrac{}{}{0pt}{}{}{-}} \mathbin{\genfrac{}{}{0pt}{}{}{\cdots}}.
\end{align*}


\begin{theorem}\label{R2_mu_nu to R2 OPRL theorem}
Let $\hat{\mathcal{R}}(x;\mu_k,\nu_{k'})$  be the continued fraction associated with the perturbations \eqref{co-recursive condition DL} and \eqref{co-dilated condition DL}. Then $\hat{\mathcal{R}}(x;\mu_k,\nu_k)$ is a pure rational spectral transformation of $\hat{\mathcal{R}}(x)$ given by
\begin{align*}
\hat{\mathcal{R}}(x;\mu_k,\nu_{k'}) \dot{=} cof(\hat{\mathbf{S}}_k(x))\hat{\mathcal{R}}(x) \dot{=} \begin{bmatrix}
	\hat{\mathcal{S}}_{22}(x)	&  -\hat{\mathcal{S}}_{21}(x)\\
	-\hat{\mathcal{S}}_{12}(x) & \hat{\mathcal{S}}_{11}(x)
\end{bmatrix} \hat{\mathcal{R}}(x),
\end{align*}
where $\hat{\mathbf{S}}_k(x)$ is as given in \Cref{transfer matrix theorem OPRL R2 pp}.
\end{theorem}


\begin{remark}
It is easy to verify that the above theorem is a generalization of results on spectral transformations for COPRL established in Section 4 of \cite{Paco perturbed recurrence 1990} {\rm (see also \cite{Castillo chapter 2017})}.
\end{remark}


\section{A prescription between co-recursion/co-dilation first and approximation of new orthogonality measure}\label{Numerical quadrature}
Since we are dealing with perturbations at different levels, it is eventual to ask whether performing co-recursion first or co-dilation first would be beneficial. From our theoretical analysis of \Cref{Structural relation and Transfer matrix}, it may seem that the order of performing perturbations holds equal merit. Our current exploration will shed light on practical scenarios where a specific perturbation, when executed first, proves to be more advantageous. In the course of this exploration, a host of additional insights, grounded in numerical findings, have come to light. These insights have been compiled in \Cref{Conclusion quadrature}.

The real zeros of ${R}_{II}$ polynomials, generated by \eqref{special R2 DL pp}, are used as key ingredients while constructing quadrature rules on the real line from ${R}_{II}$ type recurrence. However, it may happen that the introduction of a perturbation in the recurrence coefficient result in ${R}_{II}$ polynomials having complex zeros. This eventuality motivates us to look for quadrature rules on the unit circle and related approximations on the complex domain, which is beyond the scope of this manuscript. Further, it has been observed that the possibility of ${R}_{II}$ polynomials having some complex zeros arises while dealing with co-dilation for some specific values of $\nu_k$. This can also be witnessed from the expression for co-dilated ${R}_{II}$ polynomials
\begin{align*}
{\mathcal{P}}_{n+1}(x;\mu_k=0,\nu_{k'}) &= \mathcal{P}_{n+1}(x)-(\nu_{k'}-1)\lambda_{k'} (x^2+\omega^2) \mathcal{P}_{k'-1}(x)\mathcal{P}^{({k'+1})}_{n-{k'}}(x).
\end{align*}
The facts that the chain sequence property of $\{\lambda_{n}\}_{n \geq 1}$ might not be preserved after co-dilation and $(x^2+\omega^2)$ in the above expression has complex zeros are altogether responsible for ${\mathcal{P}}_{n+1}(x;\mu_k=0,\nu_{k'})$ to have some (or all) complex zeros under certain situations. However, it can be seen from the following expression for co-recursive ${R}_{II}$ polynomials that this is not the case when we deal with co-recursion only ($\nu_{k'}=1$).
\begin{align*}
{\mathcal{P}}_{n+1}(x;\mu_k,\nu_{k'}=1) &= \mathcal{P}_{n+1}(x) - \mu_k\rho_k \mathcal{P}_k(x)\mathcal{P}^{(k+1)}_{n-k}(x),
\end{align*}
Hence, we bifurcate the two eventualities and examine the corresponding results separately. First, we will illustrate the implications of co-recursion and then the co-dilation aspect is scrutinized which also leads to an interesting open problem. At the end, the situation when both co-recursion and co-dilation occur simultaneously but at different levels is dealt with in. 

The quadrature rule from ${R}_{II}$ type recurrence \eqref{special R2 DL pp} derived in \cite[Theorem 2]{Bracciali Pereira ranga 2020} is stated as
\begin{theorem}
Let $x^{(n)}_j$, $j=1,\ldots,n$ be the zeros of the ${R}_{II}$ polynomial $\mathcal{P}_{n}(x)$ and $w^{(n)}_j$ be the positive weights at $x^{(n)}_j$ given by
\begin{align}\label{weight from zeros original}
	w^{(n)}_j=\dfrac{((x^{(n)}_j)^2+1)^{n-1}\lambda_{1}\ldots\lambda_{n-1}M_0}{\mathcal{P}'_{n}(x^{(n)}_j)\mathcal{P}_{n-1}(x^{(n)}_j)}, \quad j=1,\ldots,n.
\end{align}
Then, for any $f$ such that $(x^2+1)^n f(x) \in \mathbb{P}_{2n-1}$, there holds the quadrature rule
\begin{align}\label{Quadrature rule R2 original}
\int_{-\infty}^{\infty}f(x)d\varphi(x)=\sum_{j=1}^{n}w^{(n)}_j f(x^{(n)}_j),
\end{align}
where $\varphi$ is bounded non-decreasing function on $(-\infty,\infty)$ such that
\begin{align}\label{Orthogonality condition original}
	\int_{-\infty}^{\infty}x^j\dfrac{\mathcal{P}_{n}(x)}{(x^2+1)^n}d\varphi(x) = 0, \quad j=1,\ldots,n-1.
\end{align}
\end{theorem}
Note that for a given $f(x)$ and $\varphi(x)$, the right-hand side of the formula \eqref{Quadrature rule R2 original} necessitates the knowledge of the nodes $x^{(n)}_j$, the weights at $x^{(n)}_j$ and the values of $f$ at $x^{(n)}_j$. If the exact form of the orthogonality measure $\varphi(x)$ is known, the value of $w^{(n)}_j$ can be obtained directly from $\varphi(x)$. A significant problem faced while dealing with perturbations \eqref{co-recursive condition DL} and \eqref{co-dilated condition DL} is that the exact form of the new measure of orthogonality, say $\varphi^*(x)$, is not known. But, using the zeros of the first kind generalized co-polynomials of $R_{II}$ type as nodes, a formula analogous to \eqref{weight from zeros original} involving the first kind generalized co-polynomials of $R_{II}$ type can be written. We denote these new weights by $w^{(n)*}_j$ and are given by
\begin{align}\label{weight from zeros perturbed}
	w^{(n)*}_j=\dfrac{((x^{(n)*}_j)^2+\omega^2)^{n-1}\nu_{k'}\lambda_{1}\ldots\lambda_{n-1}M_0}{\mathcal{P}'_{n}(x^{(n)*}_j,\mu_k,\nu_{k'})\mathcal{P}_{n-1}(x^{(n)*}_j,\mu_k,\nu_{k'})}, \quad j=1,\ldots,n,
\end{align}
where $x^{(n)*}_j$ are the zeros of the first kind generalized co-polynomials of $R_{II}$ type. Then, for the same $f$, the quantity $I^*_n$ given by
\begin{align}\label{Quadrature rule R2 perturbed}
	I^*_n=\sum_{j=1}^{n}w^{(n)*}_j f(x^{(n)*}_j)
\end{align}
can be computed. We will utilize \eqref{Quadrature rule R2 perturbed} to conduct numerical experiments and establish that co-recursion must be executed first, as elaborated towards the end of this section.

Let $\rho_{n}=1$, $n\geq 0$, $c_{n}=0$, $n\geq 0$, $\omega=1$ and $\lambda_{n}=1/4$, $n\geq 1$, in \eqref{special R2 DL pp}. Then, the polynomials generated by the recurrence
\begin{align}\label{Special R2 with c_n=0 DL pp}
	&\mathcal{P}_{n+1}(x) = x\mathcal{P}_n(x)-\dfrac{1}{4} (x^2+1)\mathcal{P}_{n-1}(x), \quad n\geq 1, \\
	& \mathcal{P}_{0}(x)=1, \quad \mathcal{P}_{1}(x)=x, \nonumber
\end{align}
are given by
\begin{align*}
	\mathcal{P}_{n}(x)= i\left(\dfrac{x-i}{2}\right)^{n+1}-i\left(\dfrac{x+i}{2}\right)^{n+1}, \quad n \geq 0.
\end{align*}
They are orthogonal with respect to measure $d\varphi(x)=\dfrac{1}{\pi(x^2+1)}dx$. This original weight is plotted in \Cref{1}(a). In this case, the weights $w^{(n)}_j$ have the exact value $w^{(n)}_j=\dfrac{1}{n+1}$ and consequently, the quadrature formula \eqref{Quadrature rule R2 original} becomes
\begin{align}\label{Quadrature rule from example}
I=	\int_{-\infty}^{\infty}f(x)\dfrac{1}{\pi(x^2+1)}dx=\dfrac{1}{n+1}\sum_{j=1}^{n} f(x^{(n)}_j)=I_n.
\end{align}
As an application of the quadrature rule \eqref{Quadrature rule from example}, the estimation of the integral
\begin{align*}
	I = \int_{-\infty}^{\infty}\dfrac{e^{-x^2}}{(x^2+1)^8} dx
\end{align*}
is carried out in \cite[Example 3]{Bracciali Pereira ranga 2020} by letting $f(x)=\dfrac{\pi e^{-x^2}}{(x^2+1)^7}$. It is shown that $I_n \rightarrow I$ as $n$ increases. The exact value of $I$ up to 13 significant digits is $E=0.6133229495946$. The values of $I^*_4$, $I^*_6$, $I^*_8$, $I^*_{10}$, $I^*_{12}$ and $I^*_{15}$ for $\mu_0=10^{-1}$, $10^{-2}$, and $10^{-3}$ (perturbation at $k=0$) are tabulated in \Cref{T1 DL}. The weight functions involved in computing $I^*_{10}$ for $\mu_0=10^{-1}$, $10^{-2}$, and $10^{-3}$ are plotted in \Cref{1}(b). The values of $I^*_{15}$ presented in \Cref{T2 DL} are calculated by changing the level of perturbation $k$, i.e., $k=3$, $5$, $10$, and $14$. The graphs of weight functions associated with different levels of perturbation are plotted in \Cref{2}(a). We have assumed $\nu_{k'}=1$ in all computations as we deal with co-recursion first. 

\begin{table}[tbh!]
	\caption{The estimates $I^*_n$ for different values of $n$ and varying $\mu_0$}
	\renewcommand{\arraystretch}{1.2}
	\label{T1 DL}
	\centering
	\begin{tabular}{|p{.5cm}|p{3cm}|p{3cm}|p{3.2cm}|}
		\hline
		\hfill	$n$ &\hfill $I^*_n$ for $\mu_0=0.1$ &\hfill  $I^*_n$ for $\mu_0=0.01$ &\hfill  $I^*_n$ for $\mu_0=0.001$\\
		\hline
		\hfill	4 &\hfill 0.5444480269  &\hfill 0.5602509406 &\hfill 0.5604131735\\
		\hline
		\hfill	6 &\hfill 0.5943944967  &\hfill 0.6123471475 &\hfill 0.6125313638\\
		\hline
		\hfill	8 &\hfill 0.5954014863  &\hfill 0.6133845488 &\hfill 0.6135690747\\
		\hline
		\hfill	10 &\hfill 0.5954005349  &\hfill 0.6133835886 &\hfill 0.6136058817\\
		\hline
		\hfill	12 &\hfill 0.5954003859  &\hfill 0.6133834386 &\hfill 0.6135679632\\
		\hline
		\hfill	15 &\hfill 0.5954003690   &\hfill 0.6133834218 &\hfill 0.6135679463\\
		\hline
	\end{tabular}
\end{table}

\begin{table}[tbh!]
	\caption{The estimates $I^*_n$ obtained on varying the perturbation level $k$ for fixed values of $\mu_k$}
	\renewcommand{\arraystretch}{1.2}
	\label{T2 DL}
	\centering
	\begin{tabular}{|p{.5cm}|p{3cm}|p{3cm}|}
		\hline
		\hfill	$k$ &\hfill $I^*_{15}$ for $\mu_k=0.1$ &\hfill $I^*_{15}$ for $\mu_k=0.01$ \\
		\hline
		\hfill	3 &\hfill 0.6153188745 &\hfill 0.6135874745 \\
		\hline
		\hfill	5 &\hfill 0.6136732043 &\hfill 0.6135708551 \\
		\hline
		\hfill	10 &\hfill 0.6135698123  &\hfill 0.6135698116 \\
		\hline
		\hfill	12 &\hfill 0.6135698113  &\hfill 0.6135698114 \\
		\hline
		\hfill	14 &\hfill 0.6135698110   &\hfill 0.6135698114 \\
		\hline
	\end{tabular}
\end{table}

\begin{figure}[t]
	\centering
	\subfigure[]{\includegraphics[scale=0.5]{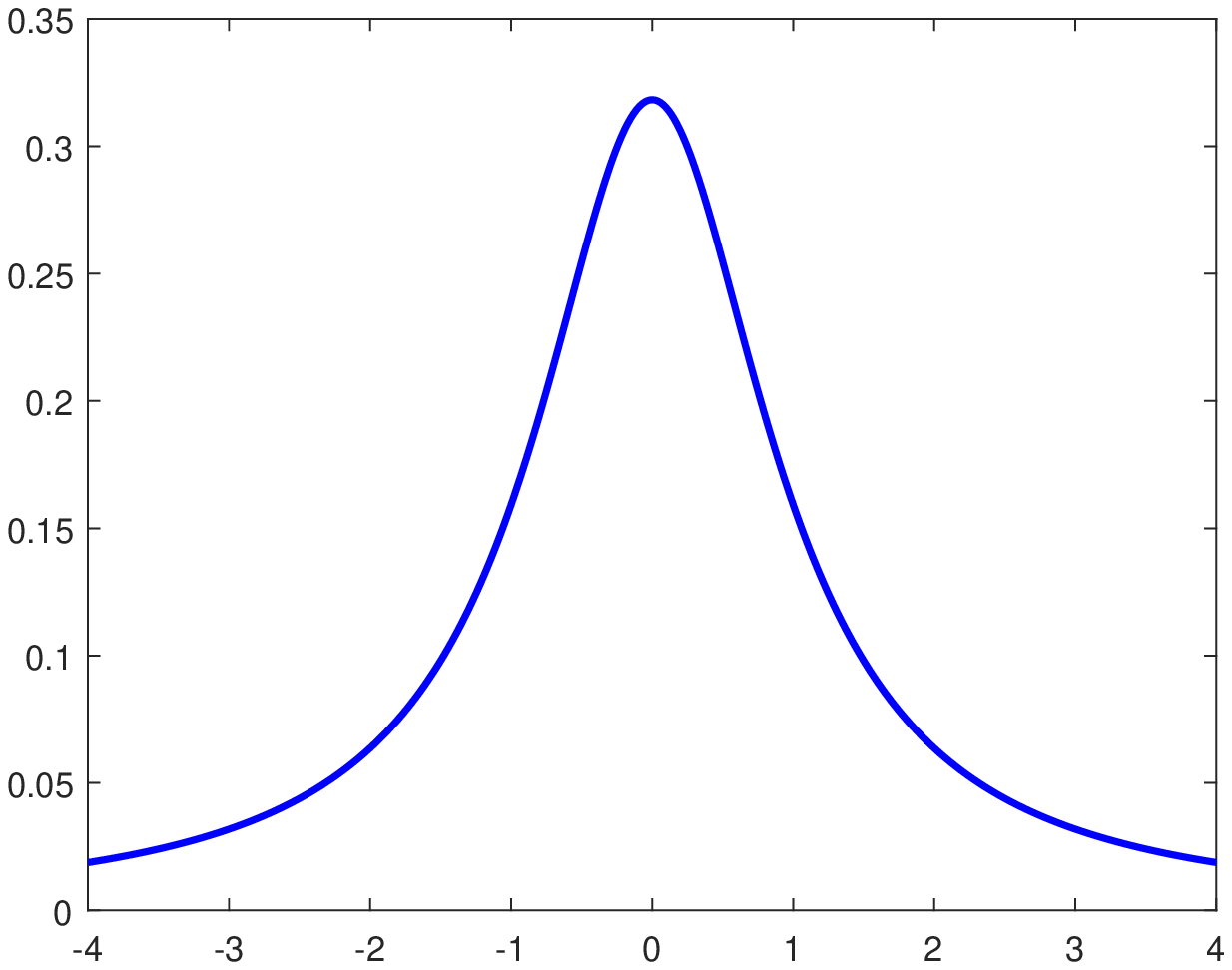}}
	\subfigure[]{\includegraphics[scale=0.5]{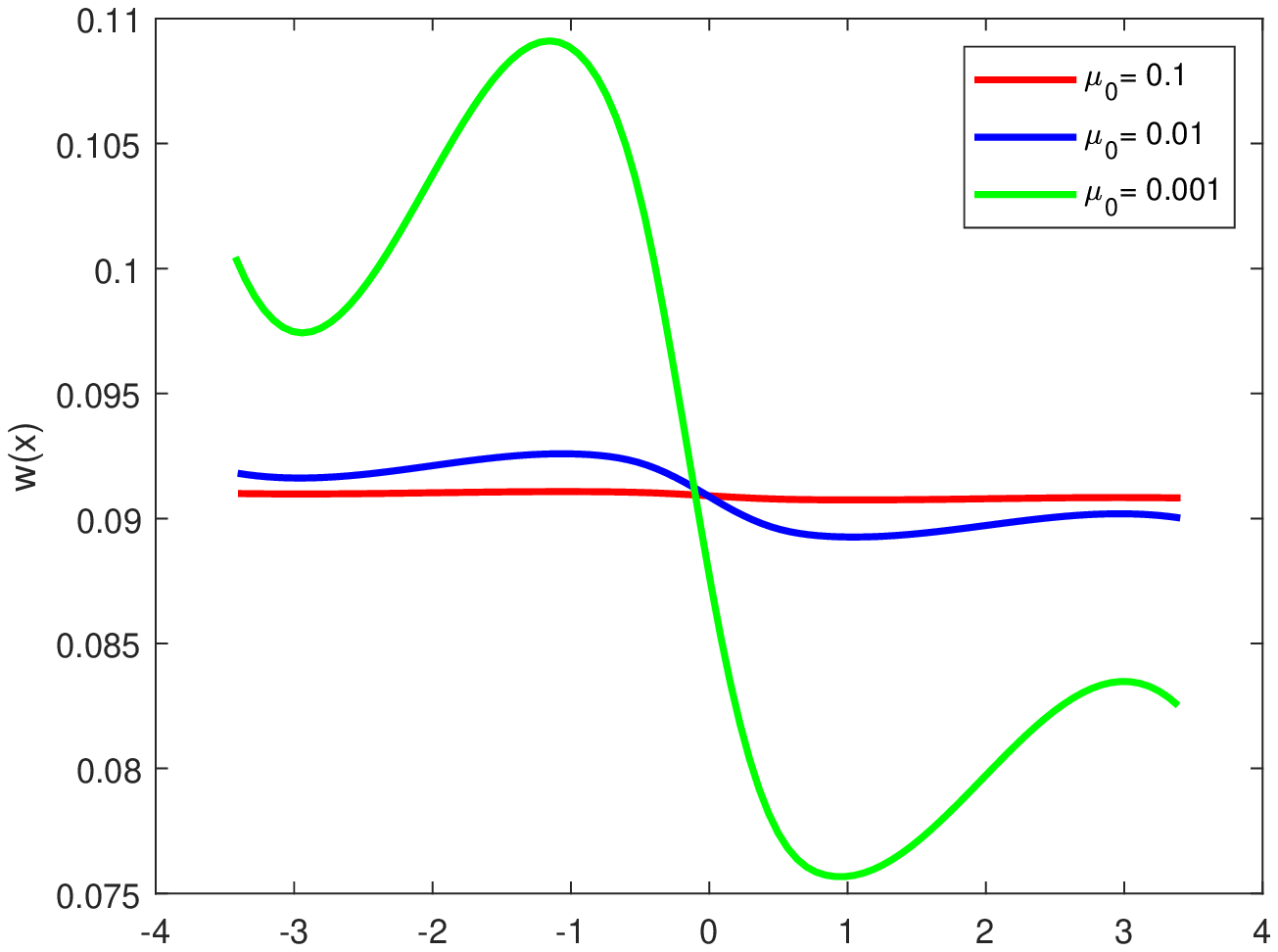}}
	\caption{(a) Graph of $\varphi(x) = \frac{1}{\pi(x^2+1)}$.~ (b) Effect of co-recurion on $\varphi(x)$ for different values of $\mu_0$}\label{1}
\end{figure}

\begin{figure}[h]
	\centering
	\subfigure[]{\includegraphics[scale=0.5]{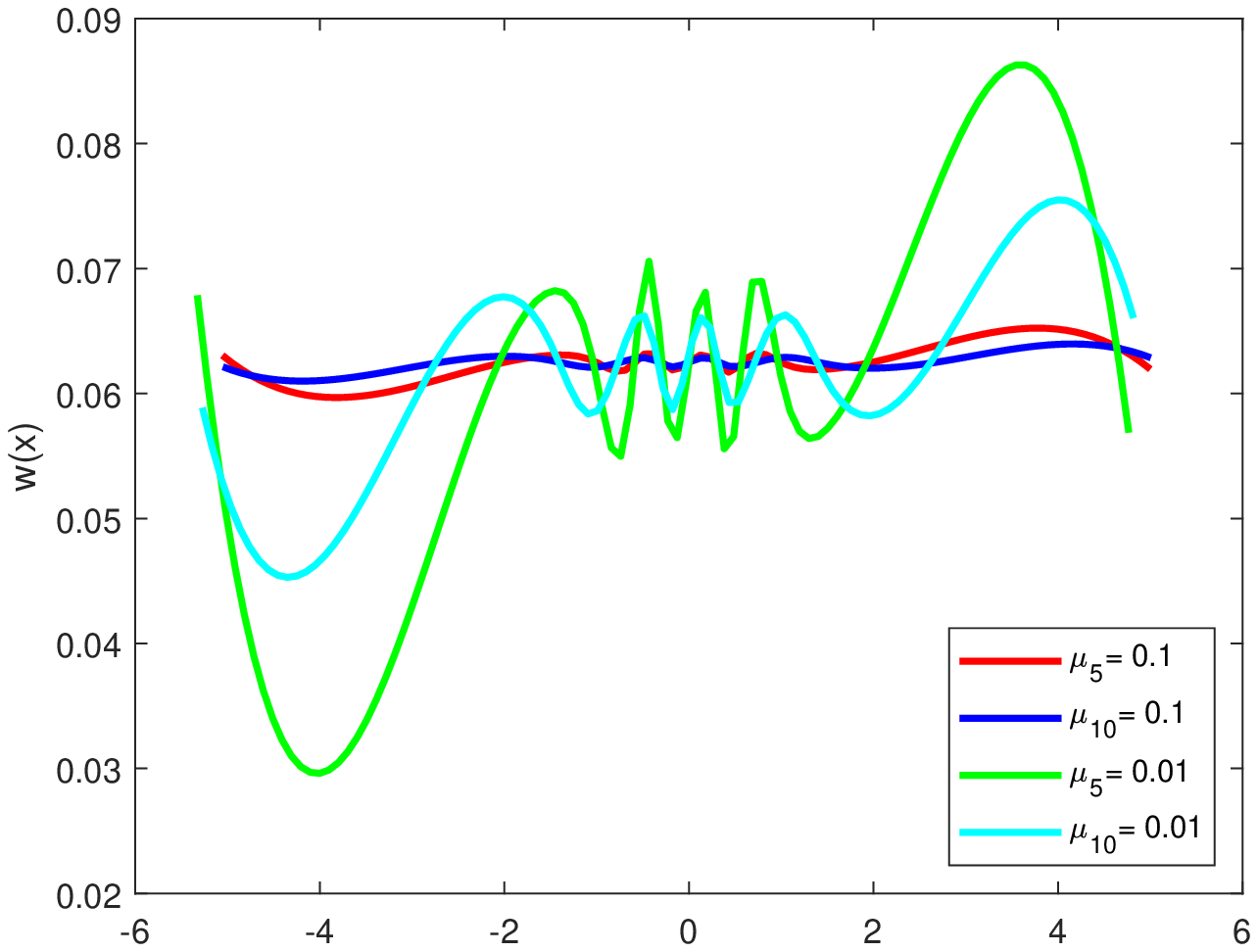}}
	\subfigure[]{\includegraphics[scale=0.5]{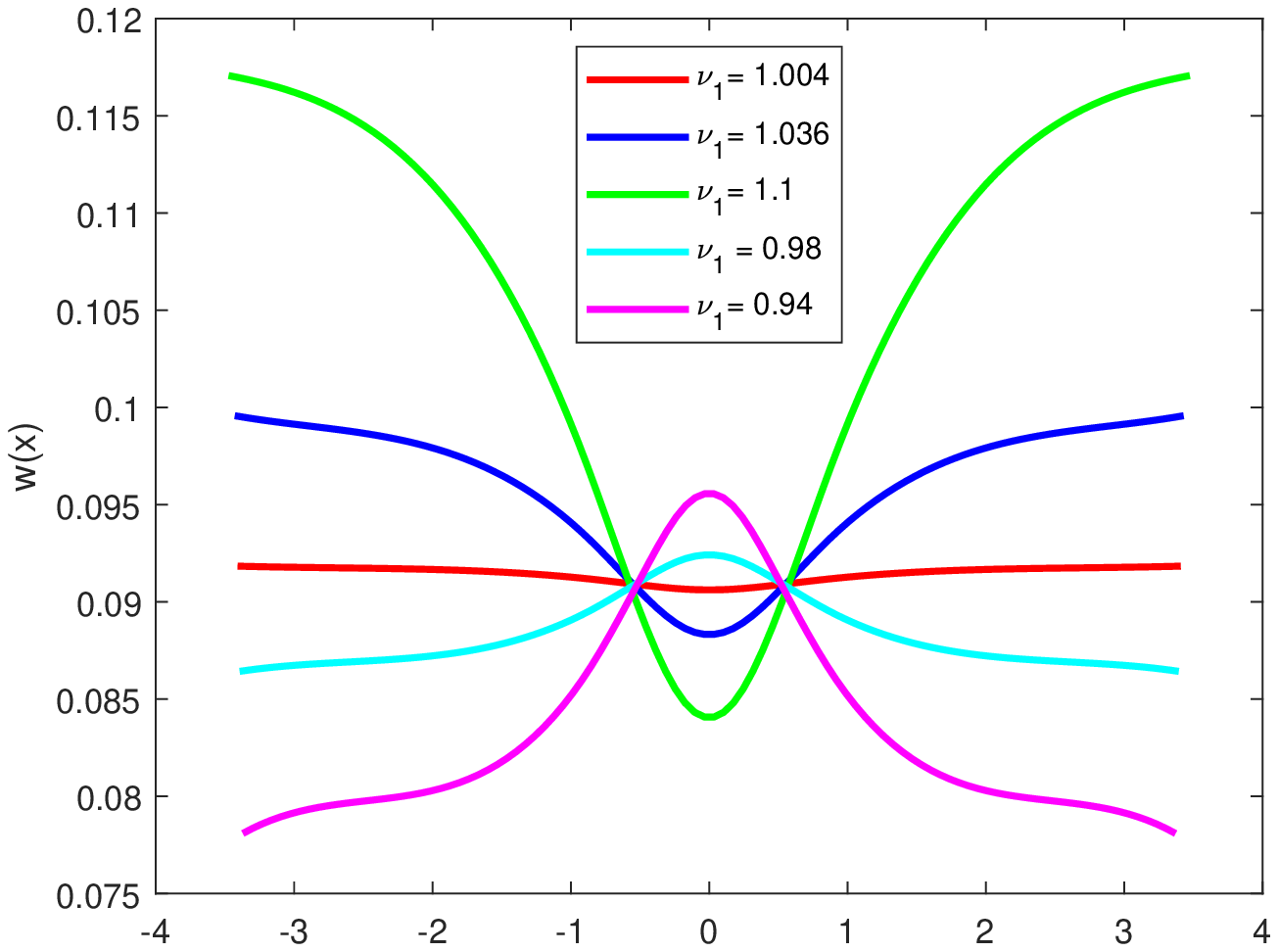}}
	\caption{(a) Effect of co-recurion at differend levels on $\varphi(x)$ for different values of $\mu_k$. ~ (b) Effect of co-dilation on $\varphi(x)$ for different values of $\nu_1$}\label{2}
\end{figure}

\begin{figure}[h]
	\centering
	\subfigure[]{\includegraphics[scale=0.5]{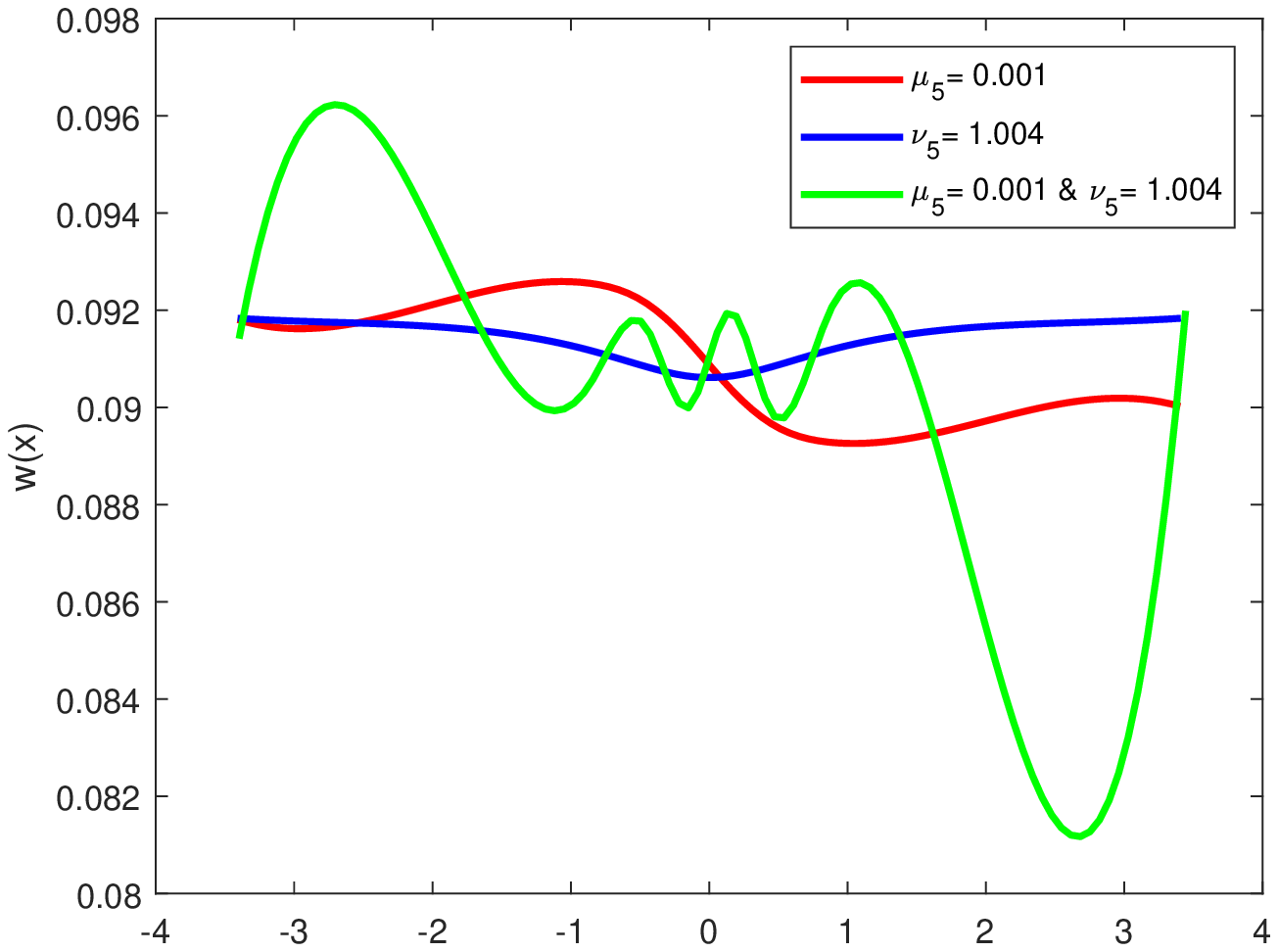}}
	\subfigure[]{\includegraphics[scale=0.5]{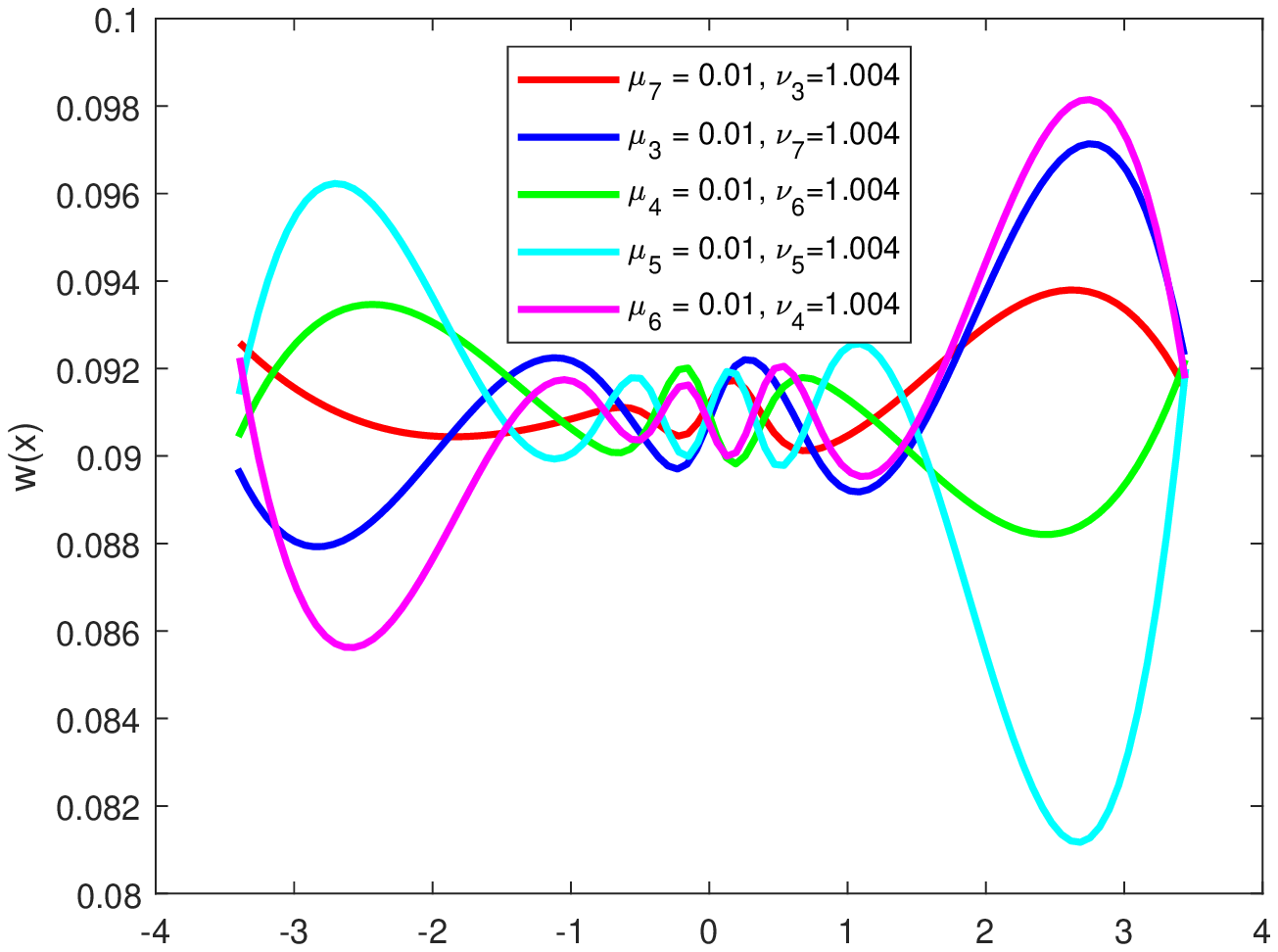}}
	\caption{(a) Comparison of weight functions corresponding to co-recursion, co-dilation and co-modification. ~(b) Effect of co-modification on $\varphi(x)$ when perturbation levels are flipped.}\label{3}
\end{figure}

Next, we find the values of $I^*_n$ when the first recurrence coefficient $\lambda_{1}$ in \eqref{Special R2 with c_n=0 DL pp} is slightly tweaked choosing $\nu_1$ close to $1$ (see \Cref{T3 DL}). One has to be careful while selecting a particular value for $\nu_1$ as for several choices, e.g., $\nu_1=2.12$, $2.16$, $2.4$, $2.6$, $2.8$, $2.96$, $3.08$, $3.28$, $3.48$, $3.64$, $3.84$, $3.96$ etc, the polynomials $\mathcal{P}_{n}(x)$ given by \eqref{Special R2 with c_n=0 DL pp} may exhibit some complex zeros. Another problem that arises while randomly selecting a $\nu_1$ is the determination of $M_0$ used in \eqref{weight from zeros perturbed} as the chain sequence structure of $\{\lambda_{n}\}_{n \geq 1}$ gets disturbed which makes it difficult to determine the minimal and maximal parameter sequences. Therefore, in the case of co-dilation, the weights $w^{(n)*}_j$ in \eqref{Quadrature rule R2 perturbed} are generated using the expression
\begin{align*}
	w^{(n)*}_j = \dfrac{\mathcal{Q}_{n}(x^{(n)*}_j;\mu_k,\nu_{k'})}{\mathcal{P}'_{n}(x^{(n)*}_j;\mu_k,\nu_{k'})}.
\end{align*}
Thus, the weights involved in computing the fourth row of \Cref{T3 DL} are plotted in \Cref{2}(b). 

\begin{table}[tbh!]
	\caption{The estimates $I^*_n$ for different values of $n$ and varying $\nu_1$}
	\renewcommand{\arraystretch}{1.2}
	\label{T3 DL}
	\centering
{\footnotesize	\begin{tabular}{|p{1cm}|p{2.4cm}|p{2.4cm}|p{2.4cm}|p{2.4cm}|p{2.4cm}|}
		\hline
		\hfill $n$ $\backslash$ $I^*_n$  &\hfill $\nu_1=0.94$ &\hfill $\nu_1=0.98$ &\hfill $\nu_1=1.004$ &\hfill $\nu_1=1.036$ &\hfill $\nu_1=1.1$\\
		\hline
		4 &\hfill 0.5922947288  &\hfill 0.5707857830 &\hfill 0.5583704426 &\hfill 0.5423646516 &\hfill 0.5121354042\\
		\hline
		6 &\hfill 0.6424741784  &\hfill 0.6222722842 &\hfill 0.6106135258 &\hfill 0.5955847852 &\hfill 0.5672031616\\
		\hline
		8 &\hfill 0.6434498760  &\hfill 0.6232893334 &\hfill 0.6116553732 &\hfill 0.5966596836 &\hfill 0.5683441160\\
		\hline
		10 &\hfill 0.6434490148  &\hfill 0.6232884038 &\hfill 0.6116544042 &\hfill 0.5966586596 &\hfill 0.5683429686\\
		\hline
		12 &\hfill 0.6434488798  &\hfill 0.6232882608 &\hfill 0.6116542536 &\hfill 0.5966584988 &\hfill 0.5683427914\\
		\hline
		15 &\hfill 0.6434488636   &\hfill 0.6232882436 &\hfill 0.6116542369 &\hfill 0.5966584821 &\hfill 0.5683427720\\
		\hline
	\end{tabular}  }
\end{table} 

Now, we will look at the case when both co-recursion and co-dilation are done simultaneously. To comprehend the combined impact of co-recursion and co-dilation on the original weight function, their joint effects are plotted alongside individual ones in \Cref{3}(a). The following tables (\Cref{T5 DL} and \Cref{T6 DL}) exhibit the values of $I^*_{10}$, for which the zeros are obtained from polynomials generated by the recurrence having both co-recursion and co-dilation simultaneously but at different levels. At first, the values and levels of co-recursion and co-dilation are fixed and value of $I^*_{10}$ is found (See row 1 and row 3 in \Cref{T5 DL} and \Cref{T6 DL}). Next, the values are kept fixed but the level of co-recursion and co-dilation is flipped to get a new value of $I^*_{10}$ (See row 2 and row 4 in \Cref{T5 DL} and \Cref{T6 DL} and compare it with row 1 and row 3 of \Cref{T5 DL} and \Cref{T6 DL}, respectively). The graphs of related weights for the second column of \Cref{T5 DL} are plotted in \Cref{3}(b). We would like to reiterate that we have avoided the values of $\mu_k$ and $\nu_{k'}$ for which complex zeros may occur.

\begin{table}[tbh!]
	\caption{The estimates $I^*_{10}$ and error from E obtained before and after interchanging the order of perturbation when co-recursion = 0.01 and co-dilation = 1.004}
	\renewcommand{\arraystretch}{1.2}
	\label{T5 DL}
	\centering
	\begin{tabular}{|p{4.2cm}|p{3cm}|p{3cm}|}
		\hline
		\hfill	Perturbation levels & \centering $I^*_{10}$ & \qquad $|I^*_{10}-E|$ \\
		\hline
		\hfill	$\mu_3=0.01$, $\nu_7=1.004$  &\hfill 0.6135307050 &\hfill 0.0002077555 \\
		\hline
		\hfill	$\mu_7=0.01$, $\nu_3=1.004$ &\hfill 0.6147312517   &\hfill 0.0014083022 \\
		\hline
		\hfill	$\mu_4=0.01$, $\nu_6=1.004$ &\hfill 0.6135500267 &\hfill 0.0002270772 \\
		\hline
		\hfill	$\mu_6=0.01$, $\nu_4=1.004$ &\hfill 0.6130399375  &\hfill 0.0002830120 \\
		\hline
		\hfill	$\mu_5=0.01$, $\nu_5=1.004$ &\hfill 0.6137394981  &\hfill 0.0004165486 \\
		\hline
	\end{tabular}
\end{table} 

\begin{table}[tbh!]
\caption{The estimates $I^*_{10}$ and error from E obtained before and after interchanging the order of perturbation when co-recursion = 0.1 and co-dilation = 0.98}
\renewcommand{\arraystretch}{1.2}
\label{T6 DL}
\centering
\begin{tabular}{|p{4cm}|p{3cm}|p{3cm}|}
	\hline
	\hfill	Perturbation levels & \centering $I^*_{10}$ & \qquad $|I^*_{10}-E|$ \\
	\hline
	\hfill	$\mu_3=0.1$, $\nu_7=0.98$ &\hfill 0.6092543888 &\hfill 0.0040685607 \\
	\hline
	\hfill	$\mu_7=0.1$, $\nu_3=0.98$ &\hfill 0.6077100195   &\hfill 0.0056129300 \\
	\hline
	\hfill	$\mu_4=0.1$, $\nu_6=0.98$ &\hfill 0.6155020892 &\hfill 0.0021791397 \\
	\hline
	\hfill	$\mu_6=0.1$, $\nu_4=0.98$ &\hfill 0.6163441652  &\hfill 0.0030212157 \\
	\hline
	\hfill	$\mu_5=0.1$, $\nu_5=0.98$ &\hfill 0.6121704739  &\hfill 0.0011524756 \\
	\hline
\end{tabular}
\end{table}
All the calculations are performed, and graphs are plotted using Mathematica$^{\tiny \textcircled{R}}$/ MATLAB$^{\tiny \textcircled{R}}$ with an Intel Core i3-6006U CPU @ 2.00 GHz and 8 GB of RAM. To analyze the behaviour of the weight function under different scenarios of co-recursion and co-dilation, the values of $w^{(n)*}_j$ between two nodes $x^{(n)*}_j$ and $x^{(n)*}_{j+1}$ are interpolated using cubic splines.

\subsection{Observations}\label{Conclusion quadrature}
\begin{itemize}
\item[1.] The weights $w^{(n)*}_j$ given by \eqref{weight from zeros perturbed} are all positive, as can be observed from graphical illustrations, and can also be analytically established using the techniques given in \cite{Bracciali Pereira ranga 2020,Esmail Ranga 2018,swami vinay GCRR 2022}.
\item[2.] Moving along the columns in \Cref{T1 DL} and \Cref{T3 DL}, it can be observed that the value of $I^*_{n}$ first increases upto a certain $n$ ($n=10$ in case of co-recursion and $n=8$ in case of co-dilation for the example in consideration) and then decreases. This phenomenon is called threshold effect. Some processes might have a certain threshold beyond which they start to behave differently. Error might rise up to this threshold and then decrease as the system adapts to this change. The \lq\lq Threshold Effect" refers to a phenomenon in which a particular system or process remains relatively unchanged or behaves in a certain way until a certain critical point or threshold is reached, after which the system experiences a significant and often rapid transformation and starts behaving differently.
\item[3.] Moving along the rows in \Cref{T1 DL}, it can be seen that the values of $I^*_{n}$ increase and tend to E as we decrease the value of $\mu_0$. In other words, $I^*_{n} \rightarrow E$ as $\mu_k \rightarrow 0$. In this case, while a large $\mu_0$ makes $\varphi(x)$ nearly flat, $\varphi(x)$ becomes sinusoidal for relatively smaller values of $\mu_0$ (see \Cref{1}(b)).
\item[4.] It can be seen by navigating along the rows in \Cref{T3 DL} that the values of $I^*_{n}$ are inversely related to $\nu_1$. Hence, it can be concluded that $I^*_{n} \rightarrow E$ as $\nu_k \rightarrow 1$. Co-dilation compresses and/or inverts $\varphi(x)$ depending upon the values of $\nu_1$. For $\nu_1 > 1$, the crust of $\varphi(x)$ gets inverted while it remains intact, and only compression happens for $\nu_1 < 1$ (see \Cref{2}(b)).
\item[5.] The values of $I^*_{n}$ increase and tend to $E$ as we traverse along the columns in \Cref{T2 DL}. This means that $|I^*_{15}-E|$ for perturbation at $k=5$ is greater than $|I^*_{15}-E|$ for perturbation at $k=10$ or $15$. In essence, this phenomenon implies that the estimate becomes increasingly refined as the level of perturbation $k$ approaches closer to the fixed value of $n$ in the determination of $I^*_{n}$. The process of generation of polynomials from the recurrence relation is an iterative process, and it is within this journey that we find the rationale behind the aforementioned phenomenon. The insight lies in the construction of $\mathcal{P}_{15}(x;\mu_{10})$, a pivotal component in calculating $w^{(15)*}_j$ and subsequently $I^*_{15}$. The perturbation $\mu_{10}$ at $k=10$, due to its lesser interaction with the recurrence coefficients during the computation process, exerts a comparatively milder influence on the outcome than the perturbation $\mu_5$ at $k=5$. This accounts for the reduced deviation from $E$.
\item[6.] The rule \eqref{Quadrature rule R2 perturbed} with $11$ nodes and perturbation $\mu_5=0.1$ and  $\mu_5=0.01$ gives $I^*_{11}=0.6136732585$ and $I^*_{11}=0.6135708116$. Observe that $11$-point rule with perturbation $\mu_5=0.1$ gives a better estimate of E than $15$-point rule with perturbation $\mu_0=0.1$ (See \Cref{T1 DL}). Further, $I^*_{11}$ for $\mu_5=0.01$ provides a better estimate than $I^*_{15}$ for $\mu_0=0.01$ (See \Cref{T2 DL}). The reason behind $I^*_{11}$ providing better estimates than $I^*_{15}$ in above cases is that for perturbation at $k=5$, five zeros of unperturbed ${R}_{II}$ polynomial  $\mathcal{P}_{5}(x)$ coincode with the zeros of co-recursive ${R}_{II}$ polynomial $\mathcal{P}_{11}(x;\mu_5)$ (See \cite[Proposition 2.2]{swami vinay R2 2022}) making the formula for $I^*_{11}$ exact (i.e., first five terms of  $I^*_{11}$ will be same as that of $I_{11}$ given by \eqref{Quadrature rule from example}) upto $5^{th}$ node, and hence contributing to enhanced accuracy. Thus, the perturbation level $k$ can be thought of as a switch between the two paths of zeros, one used for constructing \eqref{Quadrature rule R2 original} and another for formulating \eqref{Quadrature rule R2 perturbed}, i.e., upto $k$, the perturbed and original polynomials have no difference and hence the zeros are same, after $k$, the two kinds of polynomials start producing two different set of zeros with $k$ zeros still common.
\item[7.] The average of first four values of $I^*_{10}$ in the second column of \Cref{T5 DL} and \Cref{T6 DL} comes out to be $A_1 = 0.61371298$ and $A_2 = 0.61220266$. Surprisingly, $A_1$ and $A_2$ provide crude estimate for $I^*_{10}$ when $\mu_5=0.01$, $\nu_5=1.004$ and $\mu_5=0.1$, $\nu_5=0.98$, respectively. Thus, it can be inferred that the estimate $I^*_{n}$ for perturbation at a median level $m$, which is $\mu_5, \nu_5$ in this case that follows from $\frac{7+3}{2}=5$ or $\frac{4+6}{2}=5$, can be approximated by taking the average of all the estimates that are obtained when co-recursion and co-dilation are considered at different levels, say $k$ and $k'$, such that $\frac{k+k'}{2}=m$. 
\item[8.] \Cref{3}(a) and \Cref{3}(b) show how co-recursion and co-dilation work together. In \Cref{3}(a), there are strong slopes at the end points, while the gradient gradually changes sign in the middle. The right half of \Cref{3}(a) resembles the oscillations of a damped harmonic oscillator, whereas \Cref{3}(b) resembles the curve of the function $x \sin \frac{1}{x}$ as $n$ increases.
\item[9.] We commence by computing $I^*_{10}$ and its corresponding error from E in \Cref{T5 DL} and \Cref{T6 DL} for two distinct settings: $\mu_3=0.01$, $\nu_7=1.004$, and $\mu_3=0.1$, $\nu_7=0.98$ (depicted in row 1 of \Cref{T5 DL} and \Cref{T6 DL}), followed by $\mu_7=0.01$, $\nu_3=1.004$, and $\mu_7=0.1$, $\nu_3=0.98$ (depicted in row 2 of \Cref{T5 DL} and \Cref{T6 DL}). Similarly, the analysis extends to $\mu_4=0.01$, $\nu_6=1.004$, and $\mu_4=0.1$, $\nu_6=0.98$ (row 3 and row 4 of \Cref{T5 DL} and \Cref{T6 DL}). Notably, the absolute error values $|I^*_{10}-E|$ in \Cref{T5 DL} and \Cref{T6 DL} exhibit an elevation when co-dilation takes precedence over co-recursion. This observation suggests a more favorable approach: to prioritize co-recursion followed by co-dilation for enhanced estimation results.

\end{itemize}
\subsection{Approximation of measure of orthogonality} 
The ${R}_{II}$ polynomials $\mathcal{P}_{n}(x)$ and measure $\varphi(x)$ satisfy \eqref{Orthogonality condition original}. It is known that given a function $f(x)$, the following relation holds,
\begin{align}\label{Quadrature rule R2 original second time}
I=	\int_{-\infty}^{\infty}f(x)d\varphi(x)=\sum_{j=1}^{n}w^{(n)}_j f(x^{(n)}_j),
\end{align}
and the value of $I$ can be determined upto the desired $n$. Co-modification in \eqref{special R2 DL pp} yields a new set of ${R}_{II}$ polynomials $\mathcal{P}_{n}(x;\mu_k,\nu_{k'})$. Then, these polynomials satisfy the orthogonality relation
\begin{align}\label{Orthogonality condition perturbed}
	\int_{-\infty}^{\infty}x^j\dfrac{\mathcal{P}_{n}(x;\mu_k,\nu_{k'})}{(x^2+1)^n}d\varphi^*(x) = 0, \quad j=1,\ldots,n-1,
\end{align}
where $\varphi^*(x)$ is not known. The goal is to determine the closest approximation of this unknown measure. To achieve this, while maintaining $f(x)$ as stated above and using \eqref{weight from zeros perturbed}, values of $I^*_{n}$ can be numerically determined by varying by varying $n$. Among the available options, the preference is for the $I^*_{n}$ value that closely resembles $I$. The rationale behind making such a choice is that the new measure $\varphi^*(x)$ will be a modification of $\varphi(x)$ with a modification parameter $\mu_k$. That is, if $\mu_k=0$, the original weight function $\varphi(x)$ is obtained. Thus, for this fixed $n$ and $\mu_k$ and/or $\nu_k$, the corresponding $w^{(n)*}_j$ can be used to approximate a symbolic expression for the orthogonality measure $\varphi^*(x)$ such that 
\begin{align}\label{Quadrature rule R2 perturbed Final}
	\int_{-\infty}^{\infty}f(x)d\varphi^*(x)=\sum_{j=1}^{n}w^{(n)*}_j f(x^{(n)*}_j).
\end{align}
holds. Identifying a measure with respect to which a given polynomial sequence becomes orthogonal constitutes an inverse problem. For pertinent literature addressing such inquiries, we direct readers to \cite{Alhaidari 2020,Alhaidari Bahlouli 2021} and the references therein.

For the example considered in \Cref{Conclusion quadrature}, with a fixed $\mu_0$ and the provided $f(x)=\dfrac{\pi e^{-x^2}}{(x^2+1)^7}$, we can write
\begin{align*}
	\sum_{j=1}^{n}w^{(n)*}_j f(x^{(n)*}_j) = \int_{-\infty}^{\infty}\dfrac{\pi e^{-x^2}}{(x^2+1)^7}d\varphi^*(x) \rightarrow E,
\end{align*}
where $\varphi^*(x)$ is not known. An approximation of $\varphi^*(x)$ can be made using $w^{(n)*}_j$ and some interpolation techniques. This is possible once we have a sufficient degree of accuracy between the actual value of the integral $E$ and the one estimated by the $n$-point rule \eqref{Quadrature rule R2 perturbed}.  

As we can see from \Cref{T1 DL}, $I^*_{10}$ for $\mu_0=0.01$ is a very close to $E$ with an error of $6.1e-05$. It is preferable to use $w^{(10)*}_j$, $j=1,\ldots,10$ to approximate $\varphi^*(x)$. Such a $\varphi^*(x)$ will give a good approximation of the measure corresponding to perturbed ${R}_{II}$ polynomials  $\mathcal{P}_{n}(x;\mu_k=0.01,\nu_{k'}=1)$.
\begin{table}[tbh!]
	\caption{The zeros and corresponding weights used for aprroximating the new measure $\varphi^*(x)$}
	\renewcommand{\arraystretch}{1.2}
	\label{T4 DL}
	\centering
	\begin{tabular}{|p{.5cm}|p{3cm}|p{3cm}|}
		\hline
		\hfill	$j$ & \centering $x^{(10)*}_j$ & \qquad $w^{(10)*}_j$ \\
		\hline
		\hfill	1 &\hfill -3.407514395 &\hfill 0.09180849731 \\
		\hline
		\hfill	2 &\hfill -1.557863687 &\hfill 0.09242716579 \\
		\hline
		\hfill	3 &\hfill -0.8683395482  &\hfill 0.09255610991 \\
		\hline
		\hfill	4 &\hfill -0.4585153234  &\hfill 0.09215095566 \\
		\hline
		\hfill	5 &\hfill -0.1456009690   &\hfill 0.09135018250 \\
		\hline
		\hfill	6 &\hfill 0.1419649301   &\hfill 0.09041930515 \\
		\hline
		\hfill	7 &\hfill 0.4548790522   &\hfill 0.08965335827 \\
		\hline
		\hfill	8 &\hfill 0.8647030373   &\hfill 0.08928416826 \\
		\hline
		\hfill	9 &\hfill 1.554227130   &\hfill 0.08941996387 \\
		\hline
		\hfill	10 &\hfill 3.403877955   &\hfill 0.09002119900 \\
		\hline
	\end{tabular}
\end{table}

Using Lagrange interpolation, the following representation for $\varphi^*(x)$ is obtained using the values of $x^{(10)*}_j$ and $w^{(10)*}_j$ tabulated in \Cref{T4 DL}.
\begin{align*}
	d\varphi^*(x)&=\Big(\dfrac{3282}{36115}-\dfrac{379}{114840} x+\dfrac{6}{44669} x^2+\dfrac{3423}{1097777} x^3-\dfrac{277}{2072989} x^4-\dfrac{1006}{490049} x^5\\
	&+\dfrac{151}{3495875} x^6+\dfrac{447}{789656} x^7-\dfrac{64}{22062735} x^8-\dfrac{213}{6021043} x^9 \Big) dx.
\end{align*} 
Note that similar analysis can be carried out using various other interpolation formulas as well.
\subsection{A step further}
It is noteworthy that, corresponding to $\{\lambda_{n}=1/4\}_{n \geq 1}$ in \eqref{Special R2 with c_n=0 DL pp}, for $\nu_1=2$, we have the new chain sequence $\tilde{\lambda}_{1}=1/2$ and $\{\tilde{\lambda}_n=1/4\}_{n \geq 2}$ which is an SPPCS (Single parameter positive chain sequence), i.e., its minimal and maximal parameter sequence coincide and is given by $\ell_0=0$ and $\{\ell_{n+1}\}_{n \geq 0} = 1/2$. For details regarding SPPCS and related terminologies, we refer to \cite{swami vinay R2 2022}. In this case, the quantity 
\begin{align*}
	\mathcal{S} = 1+\sum_{n=2}^{\infty}\prod_{k=2}^{n}\dfrac{\ell_k}{1-\ell_k} 
\end{align*}
considered in \cite[Theorem 1]{Bracciali Pereira ranga 2020} is infinite and the integral $\displaystyle\int_{\mathbb{T}}\dfrac{1}{|\xi-1|^2}d\mu(\xi)$ does not exist (see \cite[Example 1]{Esmail Ranga 2018}). Recall that the assumptions made while constructing quadratue rules in \cite{Bracciali Pereira ranga 2020} are $\mathcal{S} < \infty$ and the integral $\displaystyle\int_{\mathbb{T}}\dfrac{1}{|\xi-1|^2}d\mu(\xi)$ exists. Hence, the quadrature rules framed in \cite{Bracciali Pereira ranga 2020} are not sufficient to deal with the situation discussed above. Thus, developing quadrature rules from ${R}_{II}$ type recurrence assuming $\mathcal{S} = \infty$ is an interesting open problem.

\section{Proof of Theorems 2.1, 2.2 and 2.3}\label{Proof of results}
Let us consider
	\begin{align*}
		\mathbb{P}_{n+1} &=
		\begin{bmatrix}
			\mathcal{P}_{n+1}(z) & \mathcal{P}_n(z)	
		\end{bmatrix}^T, ~~
		\mathbf{T}_n= \begin{bmatrix}
			\rho_n(z-c_n) & -\lambda_n (z-a_n)(z-b_n) \\
			1 & 0
		\end{bmatrix},\\
	& \hspace{3cm} \det(\mathbf{T}_n)=\lambda_n (z-a_n)(z-b_n).
	\end{align*}
	Now, from \eqref{special R2 DL pp}, we have
	\begin{align}\label{P_n+1 = T_n P_n}
		\mathbb{P}_{n+1}&= \mathbf{T}_n \mathbb{P}_n = \begin{bmatrix}
			\rho_n(z-c_n) & -\lambda_n (z-a_n)(z-b_n) \\
			1 & 0
		\end{bmatrix} \begin{bmatrix}
			\mathcal{P}_n(z) \\
			{\mathcal{P}}_{n-1}(z)
		\end{bmatrix},
		\end{align}
		\begin{align}\label{P_n+1 to P_0 R2 pp}
			\mathbb{P}_{n+1} &= (\mathbf{T}_n \ldots \mathbf{T}_0) \mathbb{P}_{0}, \qquad
			\mathbb{P}_{0} =
			\begin{bmatrix}
				\mathcal{P}_{0}(z) & \mathcal{P}_{-1}(z)	
			\end{bmatrix}^T.
		\end{align}
	
\noindent{\it \textbf{Proof of \Cref{Theorem s_k_x DL pp}.}}		Let us introduce
		\begin{align*}
			\mathbb{F}_{n+1}(z) & :=	\begin{bmatrix}
				\mathcal{P}_{n+1}(z)	& -\mathcal{Q}_{n+1}(z)\\
				\mathcal{P}_{n}(z)	& -\mathcal{Q}_{n}(z)
			\end{bmatrix} = \mathbf{T}_n \mathbb{F}_{n}(z).
		\end{align*}
		Clearly, $\mathbb{F}_{n+1}(z) $ can be written as the product of the transfer matrices
		\begin{align}\label{F_n+1 to T_0}
			&	\mathbb{F}_{n+1}(z) =  \mathbf{T}_n \mathbb{F}_{n}(z) = \mathbf{T}_n \ldots \mathbf{T}_{k+1}\mathbf{T}_k \mathbf{T}_{k-1} \ldots \mathbf{T}_0.
		\end{align}
		This gives
		\begin{align*}
			\det (\mathbb{F}_{n+1}(z))= \prod_{j=1}^{n}\lambda_j (z-a_j)(z-b_j),
		\end{align*}
		and hence, $\mathbb{F}_{n+1}(z) $ is non-singular. Also, we have $\mathbb{F}_{n+1}(z;\mu_k,\nu_{k'})$, the matrix containing first and second kind generalized co-polynomials of ${R}_{II}$ type, such that
		\begin{align}\label{F_n+1 mu_k nu_k}
			\mathbb{F}_{n+1}(z;\mu_k,\nu_{k'})&= \mathbf{T}_n \ldots \mathbf{T}_{k'+1}\mathbf{T}_{k'}(\nu_{k'})\mathbf{T}_{k'-1}\ldots \mathbf{T}_{k+1}\mathbf{T}_{k}(\mu_k)\mathbb{F}_{k}(z),
		\end{align}
		where
	\begin{align*}
		\mathbf{T}_k (\mu_k) & = \begin{bmatrix}
			\rho_k(z-c_k-\mu_k) & -\lambda_k (z-a_k)(z-b_k) \\
			1 & 0	
		\end{bmatrix} \quad \mbox{and} 
	\end{align*}
	\begin{align*}
		\mathbf{T}_{k'}(\nu_{k'}) & = \begin{bmatrix}
			\rho_{k'}(z-c_{k'}) & -\nu_{k'}\lambda_{k'} (z-a_{k'})(z-b_{k'}) \\
			1 & 0	
		\end{bmatrix}.
	\end{align*}
	Further, $\mathbf{T}_k (\mu_k)$ and $\mathbf{T}_{k'} (\nu_{k'})$ can be written as
	\begin{align}
		\mathbf{T}_k (\mu_k)=\mathbf{T}_k+\mathbf{M}_k,& \quad \mathbf{M}_k = \begin{bmatrix}
			-\rho_k\mu_k & 0  \\
			0 & 0	
		\end{bmatrix}, \label{T_k + M_k} \\
		\mathbf{T}_{k'}(\nu_{k'})=\mathbf{T}_{k'}+\mathbf{N}_{k'}, & \quad \mathbf{N}_k = \begin{bmatrix}
			0 & -(\nu_{k'}-1) \lambda_{k'} (z-a_{k'})(z-b_{k'}) \\
			0 & 0	
		\end{bmatrix}. \label{T_k + N_k}
	\end{align}
		From \eqref{F_n+1 mu_k nu_k}, we get
		\begin{align*}
			\mathbb{F}_{n+1}(z;\mu_k,\nu_{k'}) &= \begin{bmatrix}
				\mathcal{P}_{n+1}(z;\mu_k,\nu_{k'}) & -\mathcal{Q}_{n+1}(z;\mu_k,\nu_{k'})\\
				{\mathcal{P}}_{n}(z;\mu_k,\nu_{k'}) & -\mathcal{Q}_{n}(z;\mu_k,\nu_{k'})
			\end{bmatrix} \\ 
			&= \mathbf{T}_n \ldots \mathbf{T}_{k'+1}\mathbf{T}_{k'}(\nu_{k'})\mathbf{T}_{k'-1}\ldots \mathbf{T}_{k+1}\mathbf{T}_{k}(\mu_k) \begin{bmatrix}
				\mathcal{P}_{k}(z) & -\mathcal{Q}_{k}(z)\\
				{\mathcal{P}}_{k-1}(z) & -\mathcal{Q}_{k-1}(z)
			\end{bmatrix}.
		\end{align*}
		In view of \eqref{T_k + M_k} and \eqref{T_k + N_k}, this gives
		\begin{align*}
			&\mathbb{F}_{n+1}(z;\mu_k,\nu_k)= \mathbf{T}_n \ldots \mathbf{T}_{k'+1}(\mathbf{T}_{k'}+\mathbf{N}_{k'})\mathbf{T}_{k'-1}\ldots \mathbf{T}_{k+1}(\mathbf{T}_k+\mathbf{M}_k) \begin{bmatrix}
				\mathcal{P}_{k}(z) & -\mathcal{Q}_{k}(z)\\
				{\mathcal{P}}_{k-1}(z) & -\mathcal{Q}_{k-1}(z)
			\end{bmatrix} \\
			& =\mathbf{T}_n \ldots \mathbf{T}_{k'+1}(\mathbf{T}_{k'}+\mathbf{N}_{k'})\mathbf{T}_{k'-1}\ldots \mathbf{T}_{k+1} \begin{bmatrix}
				\mathcal{P}_{k+1}(z)-\mu_k\rho_k\mathcal{P}_{k}(z) & -\mathcal{Q}_{k+1}(z)+\mu_k\rho_k\mathcal{Q}_{k}(z)\\
				\mathcal{P}_{k}(z) & -\mathcal{Q}_{k}(z)
			\end{bmatrix}\\
			& =\mathbf{T}_n \ldots \mathbf{T}_{k'+1}(\mathbf{T}_{k'}+\mathbf{N}_{k'})
			\Biggr[ \genfrac{}{}{0pt}{}{\mathcal{P}_{k'}(z)-\mu_k\rho_k\mathcal{P}_{k}(z)\mathcal{P}^{(k+1)}_{k'-k-1}(z)}{\mathcal{P}_{k'-1}(z)-\mu_k\rho_k\mathcal{P}_{k}(z)\mathcal{P}^{(k+1)}_{k'-k-2}(z)} \\ 
			& \hspace{9cm} \genfrac{}{}{0pt}{}{-\mathcal{Q}_{k'}(z)+\mu_k\rho_k\mathcal{Q}_{k}(z)\mathcal{Q}^{(k+1)}_{k'-k-1}(z)}{-\mathcal{Q}_{k'-1}(z)+\mu_k\rho_k\mathcal{Q}_{k}(z)\mathcal{Q}^{(k+1)}_{k'-k-2}(z)} \Biggr] \\
			&=\mathbf{T}_n \ldots \mathbf{T}_{k'+1}
			\Biggr[ \genfrac{}{}{0pt}{}{\mathcal{P}_{k'+1}(z)-\mu_k\rho_k\mathcal{P}_{k}(z)\mathcal{P}^{(k+1)}_{k'-k}(z)-(\nu_{k'}-1)\lambda_{k'}\mathcal{P}_{k'-1}(z)}{\mathcal{P}_{k'}(z)-\mu_k\rho_k\mathcal{P}_{k}(z)\mathcal{P}^{(k+1)}_{k'-k-1}(z)-(\nu_{k'}-1)\lambda_{k'}\mathcal{P}_{k'-1}(z)} \\ 
			& \hspace{5cm} \genfrac{}{}{0pt}{}{-\mathcal{Q}_{k'+1}(z)+\mu_k\rho_k\mathcal{Q}_{k}(z)\mathcal{Q}^{(k+1)}_{k'-k}(z)+(\nu_{k'}-1)\lambda_{k'}\mathcal{Q}_{k'-1}(z)}{-\mathcal{Q}_{k'}(z)+\mu_k\rho_k\mathcal{Q}_{k}(z)\mathcal{Q}^{(k+1)}_{k'-k-1}(z)+(\nu_{k'}-1)\lambda_{k'}\mathcal{Q}_{k'-1}(z)} \Biggr] \\
			&=\Biggr[ \genfrac{}{}{0pt}{}{\mathcal{P}_{n+1}(z)-\mu_k\rho_k\mathcal{P}_{k}(z)\mathcal{P}^{(k+1)}_{n-k}(z)-(\nu_{k'}-1)\lambda_{k'}\mathcal{P}_{k'-1}(z)\mathcal{P}^{(k'+1)}_{n-k'}(z)}{{\mathcal{P}}_{n}(z)-\mu_k\rho_k\mathcal{P}_{k}(z)\mathcal{P}^{(k+1)}_{n-k-1}(z)-(\nu_{k'}-1)\lambda_{k'}\mathcal{P}_{k'-1}(z)\mathcal{P}^{(k'+1)}_{n-k'-1}(z)} \\ 
			& \hspace{1cm} \genfrac{}{}{0pt}{}{-\mathcal{Q}_{n+1}(z)- \mu_k\rho_k \mathcal{Q}_k(z)\mathcal{Q}^{(k+1)}_{n-k}(z)-(\nu_{k'}-1)\lambda_{k'} (z-a_{k'})(z-b_{k'}) \mathcal{Q}_{k'-1}(z)\mathcal{Q}^{({k'+1})}_{n-{k'}}(z)}{-\mathcal{Q}_{n}(z)- \mu_k\rho_k \mathcal{Q}_k(z)\mathcal{Q}^{(k+1)}_{n-k-1}(z)-(\nu_{k'}-1)\lambda_{k'} (z-a_{k'})(z-b_{k'}) \mathcal{Q}_{k'-1}(z)\mathcal{Q}^{({k'+1})}_{n-{k'}-1}(z)} \Biggr]
			\end{align*}
			which proves the theorem. \vspace{2mm}
		
		\noindent{\it \textbf{Proof of \Cref{transfer matrix theorem R2 pp}.}}
			Let $\mathbb{F}_{n+1}(z;\mu_k,\nu_{k'})$ be the polynomial matrix containing generalized co-polynomials of $R_{II}$ type as given by 
			\eqref{F_n+1 mu_k nu_k}. Then, the required relation can be expressed as
			\begin{align}\label{Proof 2.2 F_n+1 relation}
				\mathfrak{K}'(z) \mathbb{F}^T_{n+1}(z;\mu_k,\nu_{k'}) &= \mathbf{S}'_k(z) \mathbb{F}_{n+1}(z). \quad
			\end{align} 
		Now, from \eqref{F_n+1 mu_k nu_k}, we have
			\begin{align}
				\mathbb{F}_{n+1}(z;\mu_k,\nu_{k'}) &= \mathbf{T}_n \ldots \mathbf{T}_{k'+1}\mathbf{T}_{k'}(\nu_{k'})\mathbf{T}_{k'-1}\ldots \mathbf{T}_{k+1}\mathbf{T}_{k}(\mu_k) \mathbf{T}_{k-1} \ldots \mathbf{T}_0 \nonumber  \\	
				& = \mathbf{T}_n \ldots \mathbf{T}_{k'+1}\mathbf{T}_{k'} \mathbf{T}_{k'-1} \ldots \mathbf{T}_0 (\mathbb{F}_{k'+1} )^{-1}[\mathbf{T}_{k'}(\nu_{k'})\mathbf{T}_{k'-1}\ldots \mathbf{T}_{k+1}\mathbf{T}_{k}(\mu_k)]\mathbb{F}_{k}(z)\nonumber  \\
				& = \mathbb{F}_{n+1}(z) (\mathbb{F}_{k'+1}(z) )^{-1}[\mathbf{T}_{k'}(\nu_{k'})\mathbf{T}_{k'-1}\ldots \mathbf{T}_{k+1}\mathbf{T}_{k}(\mu_k)\mathbb{F}_{k}(z)]. \label{F_n-k+1}
			\end{align}
			Using \eqref{F_n-k+1}, we get
			\begin{align}
				\mathbb{F}^T_{n+1}(z;\mu_k,\nu_{k'}) &=[\mathbf{T}_{k'}(\nu_{k'})\mathbf{T}_{k'-1}\ldots \mathbf{T}_{k+1}\mathbf{T}_{k}(\mu_k) \mathbb{F}_{k}(z)]^T (\mathbb{F}_{k'+1}(z) )^{-T} \mathbb{F}^T_{n+1}(z), \label{F_n+1 transpose}
			\end{align}
			where
			\begin{align}
				\mathbb{F}^T_{k'+1}(z;\mu_k,\nu_{k'}) &= \begin{bmatrix}
					\mathcal{P}_{k'+1}(z;\mu_k,\nu_{k'})	& \mathcal{P}_{k'}(z;\mu_k)\\
					-\mathcal{Q}_{k'+1}(z;\mu_k,\nu_{k'})	& -\mathcal{Q}_{k'}(z;\mu_k)
				\end{bmatrix} = [\mathbf{T}_{k'}(\nu_{k'})\ldots \mathbf{T}_{k}(\mu_k) \mathbb{F}_{k}(z)]^T. \label{F_k+1 transpose}
			\end{align}
			Now, 
			\begin{align}
				\mathbb{F}_{k'+1}(z) &= \begin{bmatrix}
					\mathcal{P}_{k'+1}(z)	& -\mathcal{Q}_{k'+1}(z)\\
					\mathcal{P}_{k'}(z)	& -\mathcal{Q}_{k'}(z)
				\end{bmatrix}, \nonumber
			\end{align}
			and hence, by determinant formula, we get
			\begin{align}
				\det(\mathbb{F}_{k'+1}(z)) &= \prod_{j=1}^{k'}\lambda_j (z-a_{j})(z-b_{j})= \mathfrak{K}'(z), \nonumber
			\end{align}
			which means
			\begin{align}
				(\mathbb{F}_{k'+1}(z))^{-T} &= \frac{1}{\mathfrak{K}'(z)} \begin{bmatrix}
					-\mathcal{Q}_{k'}(z)	& -\mathcal{P}_{k'}(z)\\
					\mathcal{Q}_{k'+1}(z)	& \mathcal{P}_{k'+1}(z)
				\end{bmatrix}. \label{T_k F_k tranpose inverse}
			\end{align}
			Using \eqref{F_k+1 transpose} and \eqref{T_k F_k tranpose inverse}, we get \\ [2mm]
			$ \displaystyle
			[\mathbf{T}_{k'}(\nu_{k'})\mathbf{T}_{k'-1}\ldots \mathbf{T}_{k+1}\mathbf{T}_{k}(\mu_k) \mathbb{F}_{k}(z)]^T (\mathbf{T}_{k'}\mathbb{F}_{k'}(z) )^{-T}
			$
			\begin{align}
				&=  \frac{1}{\mathfrak{K}'(z)} \begin{bmatrix}
					\mathcal{P}_{k'+1}(z;\mu_k,\nu_{k'})	& \mathcal{P}_{k'}(z;\mu_k)\\
					-\mathcal{Q}_{k'+1}(z;\mu_k,\nu_{k'})	& -\mathcal{Q}_{k'}(z;\mu_k)
				\end{bmatrix} \begin{bmatrix}
					-\mathcal{Q}_{k'}(z)	& -\mathcal{P}_{k'}(z)\\
					\mathcal{Q}_{k'+1}(z)	& \mathcal{P}_{k'+1}(z)
				\end{bmatrix} \nonumber\\
				&=\frac{1}{\mathfrak{K}'(z)} \begin{bmatrix}
					\mathcal{S}'_{11}(z) & \mathcal{S}'_{12}(z)	\\
					\mathcal{S}'_{21}(z) & \mathcal{S}'_{22}(z)
				\end{bmatrix}. \label{product matrix DL pp}
			\end{align}
			Now, the first entry $\mathcal{S}'_{11}(z)$ of the matrix $\mathbf{S}'_k(z)$ can be computed as
			\begin{align*}
				&\mathcal{S}'_{11}(z)=-\mathcal{P}_{k'+1}(z;\mu_k,\nu_{k'}) \mathcal{Q}_{k'}(z)+\mathcal{P}_{k'}(z;\mu_k)\mathcal{Q}_{k'+1}(z)= -[\mathcal{P}_{k'+1}(z) - \mu_k\rho_k \mathcal{P}_k(z)\mathcal{P}^{(k)}_{k'-k}(z) \nonumber \\ 
				&-(\nu_{k'}-1)\lambda_{k'} (z-a_{k'})(z-b_{k'}) \mathcal{P}_{k'-1}(z)]\mathcal{Q}_{k'}(z)+[\mathcal{P}_{k'}(z) - \mu_k\rho_k \mathcal{P}_k(z)\mathcal{P}^{(k)}_{k'-k-1}(z)]\mathcal{Q}_{k'+1}(z) \nonumber\\
				&= \mathcal{P}_{k'}(z)\mathcal{Q}_{k'+1}(z)-\mathcal{Q}_{k'}(z)\mathcal{P}_{k'+1}(z)+\mu_{k}\rho_k\mathcal{P}_{k}[\mathcal{P}^{(k)}_{k'-k}(z)\mathcal{Q}_{k'}(z)-\mathcal{P}^{(k)}_{k'-k-1}(z)\mathcal{Q}_{k'+1}(z)] \\
				&+(\nu_{k'}-1)\lambda_{k'} (z-a_{k'})(z-b_{k'}) \mathcal{P}_{k'-1}(z)\mathcal{Q}_{k'}(z) \nonumber\\
				&= \mathfrak{K}'(z)+ \mu_{k}\rho_k\mathcal{P}_{k}\mathcal{Q}_{k}\prod_{j=k+1}^{k'}\lambda_j (z-a_{j})(z-b_{j})+(\nu_{k'}-1)\lambda_{k'} (z-a_{k'})(z-b_{k'}) \mathcal{P}_{k'-1}(z)\mathcal{Q}_{k'}(z).  \nonumber
			\end{align*}
			Similarly,
			\begin{align*}
				&\mathcal{S}'_{12}(z)=-\mathcal{P}_{k'+1}(z;\mu_k,\nu_{k'})\mathcal{P}_{k'}(z)+\mathcal{P}_{k'}(z;\mu_k)\mathcal{P}_{k'+1}(z)=-[\mathcal{P}_{k'+1}(z) - \mu_k\rho_k \mathcal{P}_k(z)\mathcal{P}^{(k)}_{k'-k}(z), \nonumber \\
				&-(\nu_{k'}-1)\lambda_{k'} (z-a_{k'})(z-b_{k'}) \mathcal{P}_{k'-1}(z)]\mathcal{P}_{k'}(z)+ [\mathcal{P}_{k'}(z) - \mu_k\rho_k \mathcal{P}_k(z)\mathcal{P}^{(k)}_{k'-k-1}(z)]\mathcal{P}_{k'+1}(z)\\
				&=\mu_{k}\rho_k\mathcal{P}_{k}[\mathcal{P}^{(k)}_{k'-k}(z)\mathcal{P}_{k'}(z)-\mathcal{P}^{(k)}_{k'-k-1}(z)\mathcal{P}_{k'+1}(z)] +(\nu_{k'}-1)\lambda_{k'} (z-a_{k'})(z-b_{k'}) \mathcal{P}_{k'-1}(z)\mathcal{P}_{k'}(z)\\
				&=\mu_{k}\rho_k\mathcal{P}^2_{k}\prod_{j=k+1}^{k'}\lambda_j (z-a_{j})(z-b_{j})+(\nu_{k'}-1)\lambda_{k'} (z-a_{k'})(z-b_{k'}) \mathcal{P}_{k'-1}(z)\mathcal{P}_{k'}(z).
			\end{align*}
			In line with previous expressions, we get $\mathcal{S}_{21}(z)$ and $\mathcal{S}_{22}(z)$ as
			\begin{align*}
				\mathcal{S}'_{21}(z)&=\mathcal{Q}_{k'+1}(z;\mu_k,\nu_{k'})\mathcal{Q}_{k'}(z)-\mathcal{Q}_{k'}(z;\mu_k)\mathcal{Q}_{k'+1}(z) \\
				&=-\mu_{k}\rho_k\mathcal{Q}^2_{k}\mathfrak{m}'(z)-(\nu_{k'}-1)\lambda_{k'} (x^2+\omega^2) \mathcal{Q}_{k'-1}(z)\mathcal{Q}_{k'}(z), \\
				\mathcal{S}'_{22}(z)&=\mathcal{Q}_{k'+1}(z;\mu_k,\nu_{k'})\mathcal{P}_{k'}(z)-\mathcal{Q}_{k'}(z;\mu_k)\mathcal{P}_{k'+1}(z) \\
				& = - \mu_{k}\rho_k\mathcal{Q}_{k}\mathcal{P}_{k}\mathfrak{m}'(z)-(\nu_{k'}-1)\lambda_{k'} (x^2+\omega^2) \mathcal{Q}_{k'-1}(z)\mathcal{P}_{k'}(z)+\mathfrak{K}'(z). \nonumber
			\end{align*}
			Substituting the above four relations, \eqref{product matrix DL pp} reduces to
			\begin{align}\label{N_k upon kappa}
				[\mathbf{T}_{k'}(\nu_{k'})\ldots \mathbf{T}_{k}(\mu_k) \mathbb{F}_{k}(z)]^T (\mathbf{T}_{k'}\mathbb{F}_{k'}(z) )^{-T} &= \frac{\mathbf{S}'_k(z)}{\mathfrak{K}'(z)} = \frac{\mathbf{S}'_k(z)}{ \prod_{j=1}^{k'}\lambda_j (z-a_{j})(z-b_{j})}.
			\end{align}
			Using \eqref{N_k upon kappa} in \eqref{F_n+1 transpose} gives
			\begin{align*}
				\left(\prod_{j=1}^{k'}\lambda_j (z-a_{j})(z-b_{j})\right) \mathbb{F}^T_{n+1}(z;\mu_k,\nu_{k'}) &= \mathbf{S}'_k(z) \mathbb{F}_{n+1}(z),
			\end{align*}
			which gives \eqref{Proof 2.2 F_n+1 relation} and the proof is complete.

			\vspace{2mm}
			
			\noindent{\it \textbf{Proof of \Cref{Spectral tranformation cofactor theorem}.}}
				Eliminating $\mathcal{R}_{II}^{k'+1}(z)$ from \eqref{R1_mu_nu to R1_k+1} and \eqref{R1_k+1 to R1_z} gives
				\begin{align*}
					&\mathcal{R}_{II}(z;\mu_k,\nu_{k'}) = \frac{\mathcal{A}(z)\mathcal{R}_{II}^{k'+1}(z)+\mathcal{B}(z)}{\mathcal{C}(z)\mathcal{R}_{II}^{k'+1}(z)+\mathcal{D}(z)} \\
					&= \dfrac{\lambda_{k'+1}(z-a_{k'+1})(z-b_{k'+1}) [\mathcal{Q}_{k'}(z) - \mu_k\rho_k \mathcal{Q}_k(z)\mathcal{Q}^{(k+1)}_{k'-k-1}(z)]\mathcal{R}_{II}^{k'+1}(z)-\mathcal{Q}_{k'+1}(z) }{\lambda_{k'+1}(z-a_{k'+1})(z-b_{k'+1}) [\mathcal{P}_{k'}(z) - \mu_k\rho_k \mathcal{P}_k(z)\mathcal{P}^{(k+1)}_{k'-k-1}(z)]\mathcal{R}_{II}^{k'+1}(z)-\mathcal{P}_{k'+1}(z) } \\
					&\hspace{4cm} \dfrac{+ \mu_k\rho_k \mathcal{Q}_k(z)\mathcal{Q}^{(k+1)}_{k'-k}(z)+(\nu_{k'}-1)\lambda_{k'} (z-a_{k'})(z-b_{k'}) \mathcal{Q}_{k'-1}(z)}{+ \mu_k\rho_k \mathcal{P}_k(z)\mathcal{P}^{(k+1)}_{k'-k}(z)+(\nu_{k'}-1)\lambda_{k'} (z-a_{k'})(z-b_{k'}) \mathcal{P}_{k'-1}(z)}\\
					&= \frac{[\mathcal{P}_{k'+1}(z)\mathcal{R}_{II}(z)-\mathcal{Q}_{k'+1}(z)][\mathcal{Q}_{k'}(z) - \mu_k\rho_k \mathcal{Q}_k(z)\mathcal{Q}^{(k+1)}_{k'-k-1}(z)]-[\mathcal{Q}_{k'+1}(z)}
					{[\mathcal{P}_{k'+1}(z)\mathcal{R}_{II}(z)-\mathcal{Q}_{k'+1}(z)][\mathcal{P}_{k'}(z) - \mu_k\rho_k \mathcal{P}_k(z)\mathcal{P}^{(k+1)}_{k'-k-1}(z)]-[\mathcal{P}_{k'+1}(z)} \\
					& \hspace{1cm}\dfrac{-\mu_k\rho_k \mathcal{Q}_{k}(z)\mathcal{Q}^{(k+1)}_{k'-k}(z) -(\nu_{k'}-1)\lambda_{k'}(z-a_{k'})(z-b_{k'})\mathcal{Q}_{k'-1}(z)][\mathcal{P}_{k'}(z)\mathcal{R}_{II}(z)-\mathcal{Q}_{k'}(z)]}{-\mu_k\rho_k \mathcal{P}_{k}(z)\mathcal{P}^{(k+1)}_{k'-k}(z) -(\nu_{k'}-1) \lambda_{k'}(z-a_{k'})(z-b_{k'}) \mathcal{P}_{k'-1}(z)][\mathcal{P}_{k'}(z)\mathcal{R}_{II}(z)-\mathcal{Q}_{k'}(z)]} \\
				&= \frac{[-\mathfrak{K}'(z)\lambda_j+\mu_k\rho_k \mathcal{Q}_{k}\mathcal{P}_{k}\mathfrak{m}'(z) +(\nu_{k'}-1)\lambda_{k'}(z-a_{k'})(z-b_{k'})\mathcal{Q}_{k'-1}\mathcal{P}_{k'}(z)]\mathcal{R}_{II}(z)}
				{[\mu_k\rho_k\mathcal{P}^2_{k}(z)\mathfrak{m}'(z)+(\nu_{k'}-1) \lambda_{k'}(z-a_{k'})(z-b_{k'}) \mathcal{P}_{k'-1}(z)\mathcal{P}_{k'}(z)] \mathcal{R}_{II}(z)} \\
				& \dfrac{-\mu_k\rho_k\mathcal{Q}^2_{k}(z)\mathfrak{m}'(z)-(\nu_{k'}-1)\lambda_{k'}(z-a_{k'})(z-b_{k'})\mathcal{Q}_{k'-1}(z)\mathcal{Q}_{k'}(z)}{- \mathfrak{K}'(z)-\mu_k\rho_k \mathcal{P}_{k}(z)\mathcal{Q}_{k}(z)\mathfrak{m}'(z)-(\nu_{k'}-1)\lambda_{k'}(z-a_{k'})(z-b_{k'})\mathcal{P}_{k'-1}(z)\mathcal{Q}_{k'}(z)} \\
				&=\dfrac{\mathcal{S}'_{22}(z)\mathcal{R}_{II}(z)-\mathcal{S}'_{21}(z)}{-\mathcal{S}'_{12}(z)\mathcal{R}_{II}(z)+\mathcal{S}'_{11}(z)}. \qedhere 
			\end{align*}
\subsection*{Acknowledgments} The work of the second author is supported by the NBHM(DAE) Project No. NBHM/RP-1/2019.

\end{document}